\documentclass{article}

\usepackage{amsmath}
\usepackage{amsthm}
\usepackage{amssymb}
\usepackage[marginal]{footmisc}
\usepackage{graphicx}
\usepackage{natbib}
\usepackage{mathrsfs}
\usepackage{titlesec}  % 加载 titlesec 包
\usepackage{enumitem}
\usepackage[colorlinks=true,linkcolor=blue,citecolor=blue,urlcolor=blue]{hyperref}

\newtheorem{theorem}{Theorem}[section]
\newtheorem{thm}[theorem]{Theorem}
\newtheorem{cor}[theorem]{Corollary}
\newtheorem{prop}[theorem]{Proposition}
\newtheorem{lem}[theorem]{Lemma}
\newtheorem{proposition}[theorem]{Proposition}
\newtheorem{lemma}[theorem]{Lemma}
\theoremstyle{definition}
\newtheorem{defn}[theorem]{Definition}
\newtheorem{definition}[theorem]{Definition}

\newtheorem{rmk}[theorem]{Remark}

\newtheorem{conj}[theorem]{Conjecture}

\newcommand{\bC}{\mathbb{C}}

\newcommand{\bR}{\mathbb{R}}

\def\db{\bar{\partial}}

\def\Herm{\text{\small  Herm}}
\newcommand{\R}{\mathbb{R}}

\newcommand{\Pol}{\operatorname{Pol}}

\newcommand{\Ga}{\mathrm{G}}
\newcommand{\Udom}{\mathrel{\gtrdot}}

\newcommand{\conv}{\operatorname{conv}}

\makeatletter
\let\savedtitle\title
\let\savedauthor\author
\let\saveddate\date
\let\savedthanks\thanks
\let\savedand\and
\let\savedmaketitle\maketitle
\let\savedatmaketitle\@maketitle
\newcommand{\resetmaketitle}{%
  \global\let\title\savedtitle
  \global\let\author\savedauthor
  \global\let\date\saveddate
  \global\let\thanks\savedthanks
  \global\let\and\savedand
  \global\let\maketitle\savedmaketitle
  \global\let\@maketitle\savedatmaketitle
}
\makeatother

\begin{document}

%\begin{figure}
 % \centering
  %% Requires \usepackage{graphicx}
  %\includegraphics[width=0.80\textwidth]{figure1}\\

%\end{figure}

\title{A numerical criterion for complex Hessian type equations on projective manifolds}

\author{%
  Gao Chen, Sijie Nie and Yulun Xu
}
%\date{\today}
\maketitle

% 摘要
\begin{abstract}
We prove a Nakai-Moishezon-type criterion for complex Hessian-type equations on projective manifolds whose associated degree-$n$ polynomials are strongly strictly right-Noetherian. For strictly right-Noetherian polynomials of arbitrary degree, we prove a uniform Nakai-Moishezon-type criterion. This class includes the complex Hessian and Hessian quotient equations.
\end{abstract}

\tableofcontents   %目录页

\section{Introduction}

The complex Monge-Amp\`ere equation entered K\"ahler geometry through Calabi's conjecture \cite{CalabiConj}, proved by Yau \cite{Yau1978OnTR}. Subsequent work developed existence theory for the J-equation, complex Hessian equations, and other fully nonlinear equations under C-subsolution hypotheses \cite{Song2004OnTC,fang2010classfullynonlinearflow,Gabor,Collins201511FW}. In a complementary direction, Demailly and P\u{a}un gave a numerical characterization of the K\"ahler cone \cite{JPDemaillyPaun}. Analogous numerical criteria for the J-equation were established by Collins and Sz\'{e}kelyhidi in the toric case \cite{CollinsGabor} and by the first author on arbitrary compact K\"ahler manifolds \cite{GaoChen2021}. Further numerical and algebro-geometric criteria have since been obtained for generalized Monge-Amp\`ere equations \cite{Datar2020ANC,ChuLeeTakhashi,song2020nakaimoishezoncriterionscomplexhessian}, including equations involving differential forms \cite{Fang2023OnAF} and the $\frac{\sigma_2}{\sigma_1}$- and $\sigma_2$-equations on K\"ahler threefolds \cite{2026arXiv260711296F}.

       Let $(M,\chi)$ be a compact connected smooth projective manifold of complex dimension $n$ with a K\"{a}hler form $\chi$, and $[\omega_0] \in H^{1,1}(M;\bR)$, where $ H^{1,1}(M;\bR)$ is the $(1,1)$-Dolbeault cohomology group.
        For $d\le n$, consider the following complex Hessian type equation:
        \begin{equation}
            \omega^{d}\wedge\chi^{n-d} = \sum_{k=0}^{d-1}a_k\omega^{k}\wedge\chi^{n-k}
        \end{equation}
        for $\omega \in [\omega_0]$.
        The associated polynomial is $f(x)=x^d-\sum_{k=0}^{d-1}a_k x^{k}$. When $d=n$, this equation is also called a Monge-Amp\`ere-type equation. Lin introduced right-Noetherian polynomials in \cite{LIN2023110038}; see Definition~\ref{right-Noetherian}. When $d=n$, Theorem~1.2 of \cite{LIN2023110038} proves the equivalence between right-Noetherian polynomials and the so-called $\Upsilon$-stable condition for the associated complex Hessian-type equation. More recently, Fang and Ma \cite{FangMaGarding} generalized this theory to polynomials with $d<n$, as well as to other multivariable polynomials. The cone condition can also be defined for $d<n$ in Definition~\ref{DefnConeCondition}.

Our first theorem studies strongly strictly right-Noetherian polynomials of degree $n$:
        \begin{thm}
 Let $(M,\chi)$ be a compact connected smooth projective manifold of complex dimension $n$. Assume that $c_k$ are real constants for $k=1,2,\ldots,n-1$, and $c_0$ is a real-valued function on $M$. If the polynomial $f(x)=x^n-\sum_{k=0}^{n-1}c_k\binom{n}{k}x^{k}$ associated with
\begin{equation}
            \omega^{n} = \sum_{k=0}^{n-1}c_k\binom{n}{k}\omega^{k}\wedge\chi^{n-k}
            \label{FirstEquation}
        \end{equation} is strongly strictly right-Noetherian at all points, and the integrability condition
\[
            \int_M\omega_0^{n} = \int_M\sum_{k=0}^{n-1}c_k\binom{n}{k}\omega_0^{k}\wedge\chi^{n-k}
\]
holds, then the following are equivalent:
            \begin{enumerate}
                \item The equation (\ref{FirstEquation}) has a unique smooth solution $\omega\in [\omega_0]$ satisfying the cone condition.                    
                \item There exists $\omega\in [\omega_0]$ satisfying the cone condition.
                \item For all subvarieties $V\subset M$ of dimension $p<n$,
                    we have 
                    \begin{align*}
                    &\int_V \frac{n!}{p!}\omega_0^{p} - \sum_{k=n-p}^{n-1}c_k \binom{n}{k}\frac{k!}{(k-n+p)!}\omega_0^{k-n+p}\wedge\chi^{n-k}\\
                    &=\frac{n!}{p!}\big(\int_V \omega_0^{p} - \sum_{k=0}^{p-1}c_{k+n-p} \binom{p}{k}\omega_0^{k}\wedge\chi^{p-k}\big)
                    > 0.
                    \end{align*}
            \end{enumerate}
            \label{MainThm1}
        \end{thm}

For strict right-Noetherian polynomials, including polynomials with lower degrees, we prove the following uniform result:

\begin{thm}\label{MainThm2}
 Let $(M,\chi)$ be a compact connected smooth projective manifold of complex dimension $n$. Assume that $a_k$ are real constants for $k=1,2,...,d-1$. Suppose that $f(x)=x^d-\sum_{k=0}^{d-1}a_k x^{k}$ is strictly right-Noetherian, and the integrability condition
\[
             \int_M\omega_0^{d}\wedge\chi^{n-d} = \int_M\sum_{k=0}^{d-1}a_k\omega_0^{k}\wedge\chi^{n-k}
\] holds. The following statements are then equivalent:
\begin{enumerate}
\item There exists a constant $\epsilon_0>0$ such that for any $p$-dimensional irreducible analytic subvariety
$V\subset M$,
\begin{equation}\label{e uniform stable}
\int_V \frac{d!\omega_0^{d-n+p}\wedge\chi^{n-d}}{(d-n+p)!}-\sum_{k=n-p}^{d-1}\frac{k!a_k\omega_0^{k-n+p}\wedge\chi^{n-k}}{(k-n+p)!}\ge \epsilon_0 \int_V \chi^{p},
\end{equation}
for all $n-d \le p <n$.
\item For any sufficiently small $\epsilon>0$, there exists $\omega_\epsilon\in [\omega_0]\cap\Upsilon_{F_\epsilon}$, where $F_\epsilon$ is the polarization of $f(x)+\epsilon$.
\end{enumerate}
\end{thm}

Combined with Corollary~3 of \cite{Gabor}, Theorem~\ref{MainThm2} yields the following uniform version of Sz\'{e}kelyhidi's conjecture on complex Hessian quotient equations on projective manifolds:
\begin{equation}\label{e quotient equation}
    \omega_{\varphi}^k \wedge \chi^{n-k} = C_{[\omega][\chi],k,l} \omega_{\varphi}^l \wedge \chi^{n-l}, \qquad \lambda[\chi^{-1} \omega_{\varphi}] \in \Gamma_k.
\end{equation}

We assume that
\begin{equation}
\label{weakerassumptionGabor1}
\int_M \omega^l \wedge \chi^{n-l}\neq0.
\end{equation}
The constant $C_{[\omega][\chi],k,l}$ is determined by
\[
\int_M \omega^k \wedge \chi^{n-k}-C_{[\omega][\chi],k,l} \int_M \omega^l \wedge \chi^{n-l}=0.
\]

\begin{cor}\label{cor 1}
   Let $(M,\chi)$ be a compact connected smooth projective manifold of complex dimension $n$, and let $\omega$ be a closed $(1,1)$-form. For any $0 <l<k\le n$, assume that $C_{[\omega][\chi],k,l}>0$. Then the following statements are equivalent:
   \begin{enumerate}
       \item (\ref{e quotient equation}) admits a smooth solution.
       \item There exists a constant $\epsilon_0>0$  such that for any $p$-dimensional irreducible analytic subvariety $V \subset M$,
       \begin{equation}
           \int_V \frac{k! \omega^{k-n+p}\wedge \chi^{n-k}}{(k-n+p)!} - C_{[\omega][\chi],k,l} \frac{l!\omega^{l-n+p}\wedge \chi^{n-l}}{(l-n+p)!}  \ge \epsilon_0 \int_V \chi^p,
       \end{equation}
       for all $n-l \le p <n$, 
       and
       \begin{equation}
       \label{weakerassumptionGabor2}
           \int_V \omega^{k-n+p}\wedge \chi^{n-k}\ge \epsilon_0 \int_V \chi^p,
       \end{equation}
       for all $n-k \le p <n-l$.
   \end{enumerate}
\end{cor}

Let us compare the original conjecture:

\begin{conj}(Sz\'{e}kelyhidi \cite{Gabor})
   Let $(M,\chi)$ be a compact connected smooth K\"ahler manifold of complex dimension $n$. Let $\omega$ be a closed $(1,1)$-form. For any $0 <l<k\le n$, the following statements are equivalent:
   \begin{enumerate}
       \item (\ref{e quotient equation}) admits a smooth solution.
       \item There exists a form $\omega_1 \in [\omega]$ satisfying $\lambda[\chi^{-1}\omega_1]\in \Gamma_k$ and  for any $p$-dimensional irreducible analytic subvariety $V \subset M$,
       \begin{equation*}
           \int_V \frac{k! \omega^{k-n+p}\wedge \chi^{n-k}}{(k-n+p)!} - C_{[\omega][\chi],k,l} \frac{l!\omega^{l-n+p}\wedge \chi^{n-l}}{(l-n+p)!}>0,
       \end{equation*}
       for all $n-l \le p <n$.
   \end{enumerate}
\end{conj}

One aspect of our result is stronger than Sz\'{e}kelyhidi's original conjecture: in proving $(2) \Longrightarrow (1)$ in Corollary~\ref{cor 1}, we do not assume a priori that there exists a form $\omega_1 \in [\omega]$ satisfying $\lambda[\chi^{-1}\omega_1]\in \Gamma_k$; we require only the weaker conditions~\eqref{weakerassumptionGabor1} and~\eqref{weakerassumptionGabor2}, together with $C_{[\omega][\chi],k,l}>0$.

Theorem~\ref{MainThm2} also yields the uniform version of Sz\'{e}kelyhidi's conjecture, as formulated by Murakami \cite{Murakami2024}, for complex $k$-Hessian equations
\begin{equation}\label{e hessian equation1}
    \omega_{\varphi}^k \wedge \chi^{n-k}= e^{F(x)} \chi^n , \,\,\, \lambda[\chi^{-1}\omega_{\varphi}]\in \Gamma_k 
\end{equation}
on projective manifolds. Here $F(x)$ is any smooth function satisfying 
\begin{equation*}
    \int_M \omega_0^k \wedge \chi^{n-k}= \int_M e^{F(x)}\chi^n.
\end{equation*}

\begin{cor}\label{cor 2}
   Let $(M,\chi)$ be a compact connected smooth projective manifold of complex dimension $n$. Let $\omega$ be a closed $(1,1)$-form. For any $0 <l<k\le n$, the following statements are equivalent:
   \begin{enumerate}
       \item There exists a form $\omega_1 \in [\omega]$ satisfying $\lambda[\chi^{-1}\omega_1]\in \Gamma_k$.
       \item (\ref{e hessian equation1}) admits a smooth solution.
       \item There exists a constant $\epsilon_0>0$  such that for any $p$-dimensional irreducible analytic subvariety $V \subset M$,
       \begin{equation*}
           \int_V  \omega^{k-n+p}\wedge  \chi^{n-k}  \ge \epsilon_0 \int_V \chi^p,
       \end{equation*}
       for all $n-k \le p <n$. 
   \end{enumerate}
\end{cor}

Let us compare the original conjecture:

\begin{conj}\label{conj mura}(Murakami \cite{Murakami2024})
   Let $(M,\chi)$ be a compact connected smooth K\"ahler manifold of complex dimension $n$. Let $\omega$ be a closed $(1,1)$-form. For any $0 <l<k\le n$, the following statements are equivalent:
   \begin{enumerate}
       \item There exists a form $\omega_1 \in [\omega]$ satisfying $\lambda[\chi^{-1}\omega_1]\in \Gamma_k$.
       \item For any $p$-dimensional irreducible analytic subvariety $V \subset M$,
       \begin{equation*}
           \int_V  (\omega+t\chi)^{k-n+p}\wedge  \chi^{n-k}  >0,
       \end{equation*}
       for all $n-k \le p <n$ and for all $t \ge 0$. 
   \end{enumerate}
\end{conj}

Our result requires projectivity and a uniform numerical condition, but no condition along a path.

In general, when $d=n$, \cite{lin2023solvabilitygeneralinversesigmak} proves that the cone condition in Theorem~\ref{MainThm2} is equivalent to the existence of a solution satisfying the cone condition. The analogous assertion is conjectured for $d<n$, but it has not yet been proved except for the complex Hessian and Hessian quotient equations.

If the projectivity assumption is removed, Zhang's counterexample \cite{Zhang2023} shows that, in general, an additional condition along a path is necessary. An alternative approach is to use geodesic and metric geometry---beginning with geodesics in the space of K\"ahler metrics \cite{ChenSpaceKahlerMetrics} and in the space of almost calibrated forms \cite{CollinsYauGeodesics,ChuCollinsLeeSpace}---together with the $d_1$-properness theory developed in \cite{ChuLeeHypercritical}, rather than relying solely on numerical criteria for subvarieties. The first two authors, together with other collaborators, are pursuing this direction.

Section~2 recalls the necessary background on Fang-Ma-G{\aa}rding polynomials and cone conditions. Section~3 proves the cone-inclusion lemmas. Section~4 proves the two uniform Sz\'{e}kelyhidi-type corollaries. Section~5 deduces Theorem~\ref{MainThm2} from Theorem~\ref{MainThm1}. Sections~6--9 prove Theorem~\ref{MainThm1}: Section~6 sets up the continuity path, Section~7 establishes mass concentration, Section~8 reduces closedness to a local extension theorem, and Section~9 proves the required extension results.

\paragraph{\textbf{Declaration on the use of AI.}}
The proofs of the cone-inclusion lemmas in Section~3 are revised versions of arguments generated by ChatGPT 5.6 Sol and subsequently checked and edited by the authors for mathematical clarity. ChatGPT 5.6 Sol was also used to identify gaps in an earlier version of the manuscript and to assist with grammar correction.

\section{Preliminaries}

All statements in this section are restated from
\cite{FangMaGarding} and \cite{LIN2023110038} in notation adapted to the present proof.

A polynomial is \emph{multi-affine} if its degree in each individual
variable is at most one.

\begin{definition}[Definition 3.1 of \cite{FangMaGarding}]
Let \[\sigma_k(x_1,\ldots,x_n)=\sum_{i_1<\cdots<i_k}x_{i_1}\cdots x_{i_k}\] be the \(k\)-th elementary symmetric
polynomial, with \(\sigma_0=1\).

If \(h(x)=\sum_{k=0}^{d}c_kx^k\) and \(n\ge d\), its
\emph{ordinary polarization} is
\[
  \Pol_n(h)(x_1,\ldots,x_n)
  =
  \sum_{k=0}^{d}c_k
  \frac{\sigma_k(x_1,\ldots,x_n)}{\binom nk}.
\]
It is the unique symmetric multi-affine polynomial satisfying
\[
  \Pol_n(h)(x,\ldots,x)=h(x).
\]

More generally, let
\(\boldsymbol{\kappa}=(\kappa_1,\ldots,\kappa_m)\in\mathbb N^m\), and
suppose \(h(x_1,\ldots,x_m)=\sum_{\alpha\le\boldsymbol{\kappa}}
c_\alpha x^\alpha\), where
\(\alpha=(\alpha_1,\ldots,\alpha_m)\),
\(x^\alpha=\prod_i x_i^{\alpha_i}\), and
\(\alpha\le\boldsymbol{\kappa}\) means
\(\alpha_i\le\kappa_i\) for every \(i\).  The
\(\boldsymbol{\kappa}\)-polarization is
\[
  \Pi^\uparrow_{\boldsymbol{\kappa}}(h)
  =
  \sum_{\alpha\le\boldsymbol{\kappa}}c_\alpha
  \prod_{i=1}^m
  \frac{\sigma_{\alpha_i}(x_{i1},\ldots,x_{i\kappa_i})}
       {\binom{\kappa_i}{\alpha_i}}.
\]
Its output is multi-affine in all the variables
\(\{x_{ij}:1\le i\le m,\ 1\le j\le\kappa_i\}\), and setting each
block \(x_{i1}=\cdots=x_{i\kappa_i}=x_i\) recovers \(h\).
\end{definition}

\begin{definition}
For $1\le k \le n$, write
\begin{equation*}
    \Gamma_k=\{\lambda \in \mathbb{R}^n: \sigma_i(\lambda)>0  \text{ for }1 \le i \le k\}.
\end{equation*}
\end{definition}

\begin{definition}[{Definitions 2.2 and 2.3} of \cite{FangMaGarding}]
For \(n\ge1\), write
\begin{align*}
  \Gamma_n^+
  &=\{x=(x_1,\ldots,x_n)\in\R^n:x_i>0\text{ for every }i\},\\
  \overline{\Gamma_n^+}
  &=\{x=(x_1,\ldots,x_n)\in\R^n:x_i\ge0\text{ for every }i\}.
\end{align*}
A set \(\Upsilon\subseteq\R^n\) satisfies the \emph{positive ray property}
\emph{(PRT)} if
\[
  \Upsilon+\overline{\Gamma_n^+}\subseteq \Upsilon.
\]
It satisfies the \emph{negative ray property} \emph{(NRT)} if, for
every \(x\in \Upsilon\), the set
\[
  (\{x\}-\overline{\Gamma_n^+})\cap \Upsilon
\]
is bounded.
\end{definition}

The following lemma is almost trivial, but it will be used later:
\begin{lem}
\label{strictPRT}
If $\Upsilon$ satisfies PRT, then $\bar\Upsilon+\Gamma_n^+\subseteq \Upsilon$.
\end{lem}
\begin{proof}
Suppose that $x^k\in \Upsilon$ converges to $x$ and that $y\in \Gamma_n^+$. Then, for sufficiently large $k$, $x+y \in x^k+\frac{1}{2}y+\Gamma_n^+ \subseteq x^k+\Gamma_n^+ \subseteq \Upsilon$.
\end{proof}

We now introduce the definition of a G{\aa}rding polynomial from \cite{FangMaGarding}. To distinguish this notion from the classical one, we call such a polynomial a Fang-Ma-G{\aa}rding polynomial. Since the original definition in Definitions~4.8 and 4.9 of \cite{FangMaGarding} is more complicated, we use another characterization here, which is equivalent to the original definition by Theorem~1.1(3) of \cite{FangMaGarding}.

\begin{definition}[Theorem 1.1(3) of
\cite{FangMaGarding}]
\label{FMGdefinition}
Let \(p\in\R[x_1,\ldots,x_n]\) be nonzero.  For a
multi-index
\(\alpha=(\alpha_1,\ldots,\alpha_n)\in\mathbb Z_{\ge0}^n\), write
\[
  \partial^\alpha p
  =
  \frac{\partial^{|\alpha|}p}
       {\partial x_1^{\alpha_1}\cdots\partial x_n^{\alpha_n}},
  \qquad |\alpha|=\alpha_1+\cdots+\alpha_n.
\]
Then \(p\) is called a \emph{Fang-Ma-G{\aa}rding polynomial} if and only if,
for every \(\alpha\) such that \(\partial^\alpha p\not\equiv0\),
\begin{enumerate}[label=\textup{(\roman*)}]
  \item the set \(\{\partial^\alpha p>0\}\) has a unique connected
  component \(\Upsilon_{\partial^\alpha p}\) satisfying PRT; and
  \item this component is contained in the corresponding components
  of all its first partial derivatives:
  \[
    \Upsilon_{\partial^\alpha p}
    \subseteq \Upsilon_{\partial_i\partial^\alpha p}
    \qquad(1\le i\le n).
  \]
\end{enumerate}
Here \(\Upsilon_h=\R^n\) when \(h\equiv c>0\), and by convention \(\Upsilon_0=\R^n\).
The component \(\Upsilon_p\) is called the \emph{Fang-Ma-G{\aa}rding component} of
\(p\), which is a generalization of the $\Upsilon$-stable component of \cite{LIN2023110038}.  We denote the class of such polynomials in \(n\) variables by
\(\Ga_n\), and put \(\Ga=\bigcup_{n\ge1}\Ga_n\).

\end{definition}

\begin{proposition}[{Proposition 10.2} of \cite{FangMaGarding}]
\label{product-Garding}
If \(p,q\in\Ga_n\), then their product \(pq\) also belongs to
\(\Ga_n\).
\end{proposition}

\begin{theorem}[{Theorem 8.13} of \cite{FangMaGarding}]
\label{kappa-polarization}
Let \(p\in\Ga_n\) have multidegree at most
\(\boldsymbol{\kappa}\), meaning
\(\deg_{x_i}p\le\kappa_i\) for every \(i\).  If its Fang-Ma-G{\aa}rding
component \(\Upsilon_p\) satisfies NRT, then
\[
  \Pi^\uparrow_{\boldsymbol{\kappa}}(p)\in\Ga.
\]
\end{theorem}

We next recall the definition introduced by Lin.

\begin{definition}[{Definitions 2.1} of \cite{LIN2023110038}]
\label{right-Noetherian}
For a nonconstant univariate polynomial \(h\), let \(r(h)\) denote its
largest real root, provided such a root exists.  If
\(\deg h=d\), its \emph{root sequence} is
\[
  R(h)=\bigl(r(h),r(h'),\ldots,r(h^{(d-1)})\bigr),
\]
provided all entries exist.  The polynomial \(h\) is
\emph{right-Noetherian} if all these entries exist
and
\[
  r(h)\ge r(h')\ge\cdots\ge r(h^{(d-1)}).
\]

The polynomial \(h\) is
\emph{strictly right-Noetherian} if all these entries exist
and
\[
  r(h)> r(h')\ge\cdots\ge r(h^{(d-1)}).
\]

The polynomial \(h\) is
\emph{strongly strictly right-Noetherian} if all these entries exist
and
\[
  r(h)>r(h')>\cdots>r(h^{(d-1)}).
\]
\end{definition}

The original paper \cite{LIN2023110038} only defines strictly right-Noetherian polynomials. In this paper, we need a stronger condition, which is called strongly strictly right-Noetherian in Definition \ref{right-Noetherian}.

The following is immediate from Section~9.2 of \cite{FangMaGarding}.
\begin{lemma}[Section 9.2 of \cite{FangMaGarding}]
\label{FangMaLinCorrespondence}
A monic univariate polynomial $p(x)\in \Ga_1$ if and only if it is right-Noetherian.
\end{lemma}

The following theorem will be used:

\begin{theorem}[{Theorem 2.1} of \cite{LIN2023110038}]
\label{log-concavity}
Let \(p\) be a right-Noetherian polynomial of degree \(n\). Then
\[
       \frac{p(x)p''(x)}{p'(x)^2} \le \frac{n-1}{n}
\]
for all \(x>r(p)\).
\end{theorem}

Next let us recall the definition of \(\Upsilon\)-dominance by \cite{LIN2023110038}:

\begin{definition}[{Definition 1.3} of \cite{LIN2023110038}]
Let \(f(x)\) and \(g(x)\) be right-Noetherian monic polynomials of degree \(n\). We write \(f\Udom g\) if
\[
       r(f^{(k)})\geq r(g^{(k)})\qquad(0\leq k\leq n-1).
\]
\end{definition}

\begin{theorem}[Theorem 10.10 of \cite{FangMaGarding}]
\label{UpsilonDominance}
Let \(f\) and \(g\) be right-Noetherian monic polynomials of the same degree \(d \le n\), and let \(F\) and \(G\) be their polarizations on $\mathbb{R}^n$. Then
\[
       g\Udom f
       \quad\Longleftrightarrow\quad
       \Upsilon_G\subseteq \Upsilon_F.
\]
\end{theorem}

The $d=n$ case of Theorem \ref{UpsilonDominance} was proved as Theorem 1.3 of \cite{LIN2023110038}.

\begin{lemma}[Lemma 6.6(3) of \cite{FangMaGarding}]
\label{derivative-dominance}
Let \(f\) and \(g\) be right-Noetherian monic polynomials of degrees \(d_f\le n\) and \(d_g\le n\), and let \(F\) and \(G\) be their polarizations. Then
\[
       \Upsilon_G\subseteq \Upsilon_F
       \quad\Longrightarrow\quad
       \Upsilon_{\partial_I G} \subseteq \Upsilon_{\partial_I F}
\]
for every subset \(I\subseteq\{1,2,\ldots,n\}\), where
\(\partial_I=\prod_{i\in I}\partial_i\).
\end{lemma}
       
        We now define the cone condition for equation~(\ref{FirstEquation}).
        \begin{definition}
            Let $f(x) = x^d - \sum_{k=0}^{d-1}c_k x^k$ be a right-Noetherian polynomial of degree $d\le n$, and let $F(\boldsymbol{x})$ be its polarization on $\mathbb{R}^n$. We define the $\Upsilon^k$ cone by
            \[\Upsilon^k_F=\bigcap_{|\alpha|=k} \Upsilon_{\partial^\alpha F}.\]
            If \(x \in \Upsilon^1_F\), we say that $x$ satisfies the cone condition. If $\chi$ is a given K\"ahler form, then we say that $\omega$ satisfies the cone condition if the eigenvalues $\lambda\!\left[\chi^{-1}\omega\right]$ satisfy the cone condition. When $d=n$, we also say that $\omega$ is a C-subsolution if $\omega$ satisfies the cone condition.
            \label{DefnConeCondition}
        \end{definition}

        By Section~9.2 of \cite{FangMaGarding}, the $\Upsilon^k$ cone defined here is the same as that in Definition~2.6 of \cite{LIN2023110038}.

        We also need the following results:
        \begin{prop}[Theorem 3.1 of \cite{LIN2023110038}]
            Let $f(x) = x^n - \sum_{k=0}^{n-1}c_k x^k$ be a strictly right-Noetherian polynomial, and let $F(\boldsymbol{x})$ be its polarization on $\mathbb{R}^n$. Then $\Upsilon_F$ is convex.
            \label{SeveralResult}
        \end{prop}
        Using Definition 2.6 of \cite{LIN2023110038}, it is easy to see that $\Upsilon^k_F$ is also convex for all $k$.

        The following extension lemma is essential in the proof of the main theorem:
        \begin{lemma}
        \label{extensionlemma}
        Let $f(x) = x^n - \sum_{k=0}^{n-1}c_k \binom{n}{k} x^k$ be a right-Noetherian polynomial of degree $n$, and let $F(\boldsymbol{x})=x_1\cdots x_n-\sum_{k=0}^{n-1}c_k\sigma_k(x_1,\ldots,x_n)$ be its polarization. For $m<n$, define
        \begin{align*}
        G(x_1,\ldots,x_m)
        &=\frac{\partial^{n-m}F(x_1,\ldots,x_n)}
        {\partial x_{m+1}\cdots\partial x_n}\\
        &=x_1\cdots x_m-
        \sum_{k=n-m}^{n-1}c_k\sigma_{k-n+m}(x_1,\ldots,x_m).
        \end{align*}
        Then for any $(x_1,\ldots,x_m)\in \Upsilon_G$, there exists a constant $C$ depending only on an upper bound for $|x_1|,\ldots,|x_m|$ and a positive lower bound for $\partial_I G(x_1,\ldots,x_m)$ for all $I\subseteq\{1,2,\ldots,m\}$ such that $(x_1,\ldots,x_m,C,\ldots,C)\in\Upsilon_F$. We allow $I=\emptyset$ here.
        \end{lemma}
        \begin{proof}
        By Lemma~2.5 of \cite{LIN2023110038}, it suffices to show that, for every $I\subseteq\{1,2,\ldots,n\}$,
        \[
        \partial_I F(x_1,\ldots,x_m,C,\ldots,C)>0.
        \]
        Viewed as a polynomial in $C$, this expression has leading coefficient
        \[
        \partial_{I\cap\{1,2,\ldots,m\}}G(x_1,\ldots,x_m).
        \]
        We can therefore choose $C$ sufficiently large so that the leading term controls the polynomial.
        \end{proof}

        We now study matrices through their eigenvalues. If $A$ is a positive-definite Hermitian matrix and $B$ is Hermitian, then the eigenvalues of $A^{-1}B$ are the same as those of the Hermitian matrix $A^{-\frac{1}{2}}BA^{-\frac{1}{2}}$.

        We first recall majorization. For a vector $x\in\mathbb{R}^n$, write $x^\downarrow=(x_1^\downarrow,\ldots,x_n^\downarrow)$ for its coordinates arranged in nonincreasing order. For an $A\in\Herm(n)$, write
$\lambda^\downarrow(A)=(\lambda_1^\downarrow(A),\ldots,
\lambda_n^\downarrow(A))$ for its eigenvalues in nonincreasing order.

\begin{definition}[Majorization]
Let $x,y\in\mathbb{R}^n$.  We say that $x$ is \emph{majorized} by $y$,
and write $x\prec y$, if
\[
  \sum_{i=1}^k x_i^\downarrow
  \leq
  \sum_{i=1}^k y_i^\downarrow
  \qquad (k=1,\ldots,n-1),
\]
and
\[
  \sum_{i=1}^n x_i
  =
  \sum_{i=1}^n y_i.
\]
\end{definition}

Now we state the Ky Fan--Lidskii inequality:

\begin{theorem}[Theorem~III.4.1, p.~69 of \cite{MR1477662}]
\label{Fan-Lidskii}
Let $A,B$ be Hermitian $n\times n$ matrices. Then, for $k=1,\ldots,n$,
\[
  \sum_{i=1}^k \lambda_i^\downarrow(A+B)
  \leq
  \sum_{i=1}^k \lambda_i^\downarrow(A)
  +
  \sum_{i=1}^k \lambda_i^\downarrow(B).
\]
Equality holds when $k=n$.  Equivalently,
\[
  \lambda^\downarrow(A+B)
  \prec
  \lambda^\downarrow(A)+\lambda^\downarrow(B).
\]
Consequently, for every $t\in[0,1]$,
\[
  \lambda^\downarrow\bigl(tA+(1-t)B\bigr)
  \prec
  t\lambda^\downarrow(A)+(1-t)\lambda^\downarrow(B).
\]
\end{theorem}

Now we state Rado's corollary of the Hardy--Littlewood--P\'{o}lya theorem.

\begin{theorem}[Corollary~2.B.3 of \cite{MR2759813}]
\label{Rado}
For $x,y\in\mathbb{R}^n$,  $x\prec y$ if and only if $x$ belongs to the convex hull of the coordinate permutations of $y$,
        that is,
        \[
          x\in\conv\{Py:P\text{ is a permutation matrix}\}.
        \]
\end{theorem}

As a corollary, we can improve Lin's convexity result to include matrices.

\begin{cor}
\label{convexmatrix}
Let $\Upsilon\subseteq\mathbb{R}^n$ be convex and invariant under
coordinate permutations.  Then
\[
  \mathcal{S}_{\Upsilon}
  :=\{A\in\Herm(n):\lambda^\downarrow(A)\in\Upsilon\}
\]
is convex.
\end{cor}

\begin{proof}
Take $A,B\in\mathcal{S}_{\Upsilon}$ and $t\in[0,1]$.  Put
\[
  x=\lambda^\downarrow(A),\qquad
  y=\lambda^\downarrow(B),\qquad
  w=tx+(1-t)y.
\]
Convexity of $\Upsilon$ gives $w\in\Upsilon$.  By Theorem \ref{Fan-Lidskii},
\[
  \lambda^\downarrow\bigl(tA+(1-t)B\bigr)\prec w.
\]
Theorem \ref{Rado} expresses the vector on the left as a convex combination of
coordinate permutations of $w$.  Every such permutation belongs to
$\Upsilon$, by permutation invariance, and their convex combination belongs
to $\Upsilon$, by convexity.  Hence
$tA+(1-t)B\in\mathcal{S}_{\Upsilon}$.
\end{proof}

We also need the following improvement of the positive ray property to include Hermitian matrices.

\begin{prop}
\label{PRTmatrix}
Let $\Upsilon\subseteq\mathbb{R}^n$ be invariant under coordinate
permutations and satisfy PRT. If $A, B, C, D\in\Herm(n)$, $C$ is positive definite, and $A-C$, $B$, and $D-B$ are positive semi-definite, then $\lambda^\downarrow(A^{-1}B)\in\Upsilon$ implies that \[
  \lambda^\downarrow(C^{-1}D)\in\Upsilon.
\]
\end{prop}

\begin{proof}
By the Courant--Fischer-Weyl min--max principle \cite[Chapter~III, pp.~57--83]{MR1477662},
\begin{align*}
  \lambda_i^\downarrow(A^{-1}B)
  &=\max_{\substack{V\subseteq\mathbb{C}^n\\ \dim V=i}}
    \min_{0\neq x\in V}\frac{x^*Bx}{x^*Ax}\\
  &\leq
    \max_{\substack{V\subseteq\mathbb{C}^n\\ \dim V=i}}
    \min_{0\neq x\in V}\frac{x^*Dx}{x^*Cx}
   =\lambda_i^\downarrow(C^{-1}D).
\end{align*}
\end{proof}

Now we start to study differential forms.

Let $g(t)= \sum_{i=0}^k a_i t^i$ with $k \le n$ and $a_k=1$, and set $G=\Pol_n(g)$. With respect to the fixed background form $\chi$, define
\begin{equation}
 H_G (\omega) \triangleq   \sum_{i=0}^k a_i \omega^i \wedge \chi^{n-i}.
\end{equation}
We denote by $H_G^{(l)}(\omega)$ the $l$th formal derivative of $H_G(\omega)$ with respect to $\omega$. For example, $(\omega^k \wedge \chi^{n-k})^{(l)}=\frac{k!}{(k-l)!} \omega^{k-l}\wedge \chi^{n-k}$ for $l \le k$.
We need the following lemma:
\begin{lemma}
\label{positivity}
If $\omega$ satisfies the cone condition for a right-Noetherian polynomial $g(x) = x^d - \sum_{k=0}^{d-1}a_k x^k$ with $d\le n$, then for any $n-d\le p<n$, the $(p,p)$-form
\[
H^{(n-p)}_G(\omega)=\frac{d!}{(d-n+p)!}\omega^{d-n+p}\wedge\chi^{n-d}-\sum_{k=n-p}^{d-1}\frac{k!}{(k-n+p)!}a_k\omega^{k-n+p}\wedge\chi^{n-k}
\]
is a strictly strongly positive $(p,p)$-form.
\end{lemma}

Fix a point \(x\in M\). Choose a \(\chi\)-unitary
holomorphic coframe at \(x\) in which
\[
\chi
=
\sum_{j=1}^{n}\beta_j,
\qquad
\omega
=
\sum_{j=1}^{n}\lambda_j\beta_j,
\qquad
\beta_j
=
\sqrt{-1}\,dz^j\wedge d\bar z^j.
\]
By assumption,
\[
\lambda=(\lambda_1,\ldots,\lambda_n)
=
\lambda\!\left[\chi^{-1}\omega\right](x)
\in   \Upsilon^1_G.
\]

For a subset \(I=\{i_1,\ldots,i_p\}\subset\{1,\ldots,n\}\),
write
\[
\beta_I
=
\beta_{i_1}\wedge\cdots\wedge\beta_{i_p}
\]
and let
\[
\lambda_I=(\lambda_i)_{i\in I}.
\]

Set
\[
r=n-p.
\]
Since \(n-r=p\), we have
\begin{equation}\label{e hgr1}
H_G^{(r)}(\omega)
=\frac{d!}{(d-r)!}\omega^{d-r}\wedge\chi^{n-d}-\sum_{k=n-p}^{d-1}
a_k\frac{k!}{(k-r)!}
\omega^{k-r}\wedge\chi^{n-k}.
\end{equation}

For each \(k\geq r\), direct expansion gives
\begin{equation}\label{e omega}
\omega^{k-r}\wedge\chi^{n-k}
=
(k-r)!(n-k)!
\sum_{|I|=k}
\sigma_{k-r}(\lambda_I)\beta_I.
\end{equation}

Substituting (\ref{e omega}) into (\ref{e hgr1}), we obtain
\begin{equation}\label{e hgr}
H_G^{(r)}(\omega)
=
\sum_{|I|=p}
\left[
d!(n-d)!\sigma_{d-r}(\lambda_I)-
\sum_{k=r}^{d-1}
a_k k!(n-k)!
\sigma_{k-r}(\lambda_I)
\right]\beta_I.
\end{equation}

Let \(J=I^c\), so that \(|J|=r\). Since \(G\) is
multi-affine, one has
\[
\partial_J\sigma_k(\lambda)
=
\sigma_{k-r}(\lambda_{J^c})
=
\sigma_{k-r}(\lambda_I).
\]
Consequently,
\begin{equation}\label{e gic}
G_{I^c}(\lambda)
=\frac{1}{\binom n d}
\sigma_{d-r}(\lambda_I)-
\sum_{k=r}^{d-1}
\frac{a_k}{\binom n k}
\sigma_{k-r}(\lambda_I).
\end{equation}

Comparing (\ref{e hgr}) and (\ref{e gic}), we obtain the pointwise identity
\begin{equation*}
H_G^{(n-p)}(\omega)
=
n!\sum_{|I|=p}
G_{I^c}(\lambda)\,\beta_I.
\end{equation*}
By Lemma~\ref{FangMaLinCorrespondence}, $g$ is also a \emph{Fang-Ma-G{\aa}rding polynomial}. Theorem~\ref{kappa-polarization} then shows that its polarization $G$ is also a \emph{Fang-Ma-G{\aa}rding polynomial}. By the definition of a \emph{Fang-Ma-G{\aa}rding polynomial}, $G_{I^c}(\lambda)>0$ for $\lambda \in \Upsilon^1_G$. As a result, $H_G^{(n-p)}(\omega)$ is a smooth strictly strongly positive $(p,p)$-form on $M$.

\section{The cone-inclusion lemmas}

In this section, we prove two cone-inclusion lemmas that will be used in the proofs.

\begin{lemma}\label{cone-inclusion}
Let \(f\in\R[x]\) be monic, of degree \(d\ge1\), and right-Noetherian.  Let \(T\in\R\), set
\[
  g_T(x)=(x+T)f(x),
\]
and let \(n\ge d+1\). Then $g_T$ is also right-Noetherian. Moreover, define the ordinary polarizations
\[
  F(\boldsymbol{x})=\Pol_n(f)(\boldsymbol{x}),
  \qquad
  G_T(\boldsymbol{x})=\Pol_n(g_T)(\boldsymbol{x}),
  \qquad \boldsymbol{x}=(x_1,\ldots,x_n).
\]
Then the following inclusion of Fang-Ma-G{\aa}rding components holds:
\( \Upsilon_{G_T}\subseteq \Upsilon_F. \)
\end{lemma}

\begin{proof}
First, since both $f$ and $x+T$ are right-Noetherian, i.e., in $\Ga_1$, Proposition~\ref{product-Garding} shows that their product $g_T$ is also in $\Ga_1$, i.e., right-Noetherian.

We first show that 
\[
  q_T(x,y):=f(x)(x+T+y)\in\Ga_2
\]
with NRT Fang-Ma-G{\aa}rding component.

Regard \(f(x)\) as a polynomial on \(\R^2\) with coordinates
\((x,y)\).  Its nonzero derivatives are the polynomials
\(f^{(k)}(x)\), whose unique PRT components are
\[
  (r\bigl(f^{(k)}\bigr),\infty)\times\R
  \quad(0\le k\le d-1),
  \qquad\text{and}\qquad
  \R^2\quad(k=d).
\]

The right-Noetherian condition shows directly that \(f(x)\in\Ga_2\). It is also easy to see that the affine linear polynomial
\[
  \ell_T(x,y)=x+T+y
\]
belongs to \(\Ga_2\). Proposition
\ref{product-Garding} therefore gives \(q_T(x,y)\in\Ga_2\). The Fang-Ma-G{\aa}rding component is
\[
  \Upsilon_{q_T}=\{(x,y):x>r(f),\ x+T+y>0\}.
\]

This component satisfies NRT.  Indeed, fix
\((x,y)\in \Upsilon_{q_T}\).  If \(a,b\ge0\) and
\((x-a,y-b)\in \Upsilon_{q_T}\), then
\[
  0\le a<x-r(f)
\]
and
\[
  0\le b<x-a+T+y\le x+T+y.
\]
Thus all such \((a,b)\) lie in a bounded rectangle, as required.

We now \(\boldsymbol{\kappa}\)-polarize \(q_T\). The multidegree of \(q_T\) is \((d+1,1)\le(n,1)\). Applying Theorem~\ref{kappa-polarization} with
\(\boldsymbol{\kappa}=(n,1)\) yields the multi-affine Fang-Ma-G{\aa}rding
polynomial
\[
  P_T(\boldsymbol{x},y)
  :=
  \Pi^\uparrow_{(n,1)}q_T(\boldsymbol{x},y)\in\Ga_{n+1}.
\]
By linearity of polarization and the identity
\[
  q_T(x,y)=g_T(x)+y f(x),
\]
the definition of \(\boldsymbol{\kappa}\)-polarization gives the formula
\[
  P_T(\boldsymbol{x},y)
  =
  G_T(\boldsymbol{x})+yF(\boldsymbol{x}).
\]

Thus, by the definition of Fang-Ma-G{\aa}rding polynomials, we have
\[
\Upsilon_{P_T}\subset \Upsilon_{F}\times\R,
\]
since \(F(\boldsymbol{x})=\frac{\partial}{\partial y}P_T(\boldsymbol{x},y)\).

Finally, the set \( \{ \boldsymbol{x}: (\boldsymbol{x},0)\in \Upsilon_{P_T}\} \) satisfies PRT and is therefore connected. It is easy to see that this set equals \(\Upsilon_{G_T}\), and that \(\boldsymbol{x}\in \Upsilon_{F}\) if and only if \((\boldsymbol{x},0)\in \Upsilon_{F}\times\R\). This finishes the proof.
\end{proof}

Another cone inclusion lemma is the following:

\begin{lemma}
Let \(f\in\R[x]\) be a monic univariate right-Noetherian polynomial of degree \(n\). Then for every \(\varepsilon>0\),
\[g_\varepsilon(x):=f(x)-\varepsilon f'(x)
\]
is strongly strictly right-Noetherian, and \(g_\varepsilon\Udom f\).
\label{perturbation}
\end{lemma}

\begin{proof}
We first show that the largest real root \(r(g_\varepsilon)\) exists and \(r(g_\varepsilon)>r(f)\).
To see this, near \(r(f)\), factor
\[
       f(x)=(x-r(f))^q a(x),
       \qquad q\geq1, \qquad a(r(f))>0.
\]
For sufficiently small \(\delta>0\),
\begin{align*}
 g_\varepsilon(r(f)+\delta)
 &=\delta^{q-1}\Bigl[
       \delta a(r(f)+\delta)
       -\varepsilon\bigl(q a(r(f)+\delta)
                    +\delta a'(r(f)+\delta)\bigr)
     \Bigr]
 <0.
\end{align*}
On the other hand, \(g_\varepsilon(x)\to+\infty\) as \(x\to+\infty\), because \(g_\varepsilon\) is monic. Hence, by the intermediate value theorem, \(r(g_\varepsilon)\) exists and \(r(g_\varepsilon)>r(f)\).

Using the fact that $f'$ is also right-Noetherian, and
\[(g_\varepsilon(x))'=f'(x)-\varepsilon f''(x),
\]
by induction, we see that \(r(g^{(k)}_\varepsilon)\) exists and \(r(g^{(k)}_\varepsilon)>r(f^{(k)})\) for all \(k=1,2,...,n-1\). A similar induction process shows that it suffices to prove that \(r(g_\varepsilon)>r(g'_\varepsilon)\). If \(r(g'_\varepsilon)\le r(f)\), we are done. So we only need to consider the case \(x_0:=r(g'_\varepsilon)>r(f)\).

Since \(g'_\varepsilon(x_0)=f'(x_0)-\varepsilon f''(x_0)=0\), and since \(f(x_0)>0\), \(f'(x_0)>0\), and \(f''(x_0)>0\), Theorem~\ref{log-concavity} implies that
\[
 f(x_0)\leq\frac{n-1}{n}\frac{f'(x_0)^2}{f''(x_0)}=\frac{n-1}{n} \varepsilon f'(x_0).
\]
So
\[
 g_\varepsilon(x_0) \leq \frac{-\varepsilon}{n} f'(x_0)<0.
\]
Since \(g_\varepsilon(x)\to+\infty\) as \(x\to+\infty\), \(g_\varepsilon\) has a real root strictly larger than $x_0$. Thus \( r(g_\varepsilon)>r(g'_\varepsilon)\) as required.
\end{proof}

\section{Proof of the uniform Sz\'{e}kelyhidi conjectures}
\begin{proof}
    (Proof of Corollary~\ref{cor 1}.) $(1) \Longrightarrow (2).$ Set $f(x)= x^k- C_{[\omega][\chi],k,l}x^l$, and let $F$ be its polarization. Let $\omega_{\varphi}$ be a solution of (\ref{e quotient equation}). Then $\omega_{\varphi} \in \bar\Upsilon_{F}\subset \Upsilon_{F_{\epsilon}}$, where $F_{\epsilon}$ is the polarization of $f(x)+\epsilon$. The implication ``$(2) \Longrightarrow (1)$" in Theorem~\ref{MainThm2} now gives condition~(2) of Corollary~\ref{cor 1}.

    $(2) \Longrightarrow (1)$. The implication ``$(1) \Longrightarrow (2)$" in Theorem~\ref{MainThm2} shows that, for every sufficiently small $\epsilon>0$, there exists $\omega_{\epsilon} \in [\omega]\cap \Upsilon_{F_{\epsilon}}$. For sufficiently small $\epsilon$, the root $r(f+\epsilon)$ is positive. Set $g=x^k$, and let $G$ be its polarization. Since $r(g)=0$, we have $r(f+\epsilon)>r(g)$. For $1\le i\le l$, we have
    \begin{equation*}
        r(f^{(i)})>0= r(g^{(i)}).
    \end{equation*}
    For $l+1 \le i \le k-1$, we have
    \begin{equation*}
        r(f^{(i)})= r(g^{(i)})=0.
    \end{equation*}  
    Theorem~\ref{UpsilonDominance} then gives $\Upsilon_{F_\epsilon} \subset \Upsilon_G$. Consequently,
    \begin{equation}\label{e omega ad}
        \omega_{\epsilon}\subset [\omega] \cap  \Upsilon_G= \Gamma_k.
    \end{equation} 
    Since $F_{\epsilon}$ is a \emph{Fang-Ma-G{\aa}rding polynomial}, we have that 
    \begin{equation}\label{e omega sub}
        \omega_{\epsilon} \subset \Upsilon_{F_{\epsilon}}\subset \Upsilon^1_{F_{\epsilon}} = \Upsilon^1_{F}.
    \end{equation}
    Combining (\ref{e omega ad}) and (\ref{e omega sub}), we conclude that $\omega_{\epsilon}$ is a $C$-subsolution to the complex quotient equation (\ref{e quotient equation}). The existence result in \cite{Gabor} then shows that (\ref{e quotient equation}) admits a smooth solution.
 \end{proof}

 \begin{proof}
     (Proof of Corollary~\ref{cor 2}.) $(1) \Longleftrightarrow (2)$ follows from \cite{DinewKolodziej2017}.
     
     $(2) \Longrightarrow (3)$. Set $g=x^k-C_1$, where $C_1=\frac{\int_M \omega^k \wedge \chi^{n-k}}{\int_M \chi^n}>0$, and let $G$ be the polarization of $g$. Set $g_{\epsilon}=g+\epsilon$, and let $G_{\epsilon}$ be its polarization. By (2), there is a solution $\omega_1 \in [\omega]$ satisfying
     \begin{equation*}
         \omega_1^k \wedge \chi^{n-k}= C_1 \chi^n,
         \qquad \lambda[\chi^{-1}\omega_1]\in \Gamma_k.
     \end{equation*}
     Then $\omega_1 \in \Gamma_k \cap \{ G=0\} \subset \Upsilon_{G_{\epsilon}}$. The implication ``$(2) \Longrightarrow (1)$" in Theorem~\ref{MainThm2} therefore proves $(2)\Longrightarrow (3)$.
     
     $(3) \Longrightarrow (1)$. Using the ``$(1) \Longrightarrow (2)$" direction of the Theorem \ref{MainThm2}, we can find $\omega_1 \in \Upsilon_{G_{\epsilon}} \subset \Upsilon_{\sigma^k}= \Gamma_k$.
 \end{proof}

\section{Proof of the second main theorem}
We first prove Theorem~\ref{MainThm2} assuming Theorem~\ref{MainThm1}. Let $g$ be a polynomial. Set $g_{T,l}(x)=(x+T)^l g(x)$ and $G_{T,l}(\boldsymbol{x})=\Pol_n(g_{T,l})(\boldsymbol{x})$, where $\boldsymbol{x}=(x_1,\ldots,x_n)$ and $1\le l\le n-d$.

\begin{prop}\label{prop hgtn-k}
Suppose that
\[
g(x)=x^d-\sum_{k=0}^{d-1}a_kx^k
\]
is strictly right-Noetherian and that condition~(\ref{e uniform stable}) in Theorem~\ref{MainThm2} holds for $n-d \le p \le n$. Then there exists a sufficiently large $T$ such that, for any $p$-dimensional irreducible analytic subvariety
$V\subset M$,
\begin{equation}
\int_V H_{G_{T,n-d}}^{(n-p)}(\omega) >0
\end{equation}
for every $1\le p\le n$.
\end{prop}
\begin{proof}
By definition,
\begin{equation*}
\begin{aligned}
H_{G_{T,n-d}}(\omega)
&=\omega^d\wedge(\omega+T\chi)^{n-d}\\
&\quad-\sum_{i=0}^{d-1}a_i\omega^i\wedge\chi^{d-i}
\wedge(\omega+T\chi)^{n-d}\\
&=\sum_{j=0}^{n-d}\binom{n-d}{j}T^j\chi^j\wedge\omega^{n-j}\\
&\quad-\sum_{i=0}^{d-1}\sum_{j=0}^{n-d}a_i\binom{n-d}{j}T^j
\chi^{d-i+j}\wedge\omega^{n-d-j+i}.
\end{aligned}
\end{equation*}
Set
\begin{equation*}
\begin{aligned}
\mathcal R_T(\omega)
&=\sum_{j=0}^{n-d-1}\binom{n-d}{j}T^j\chi^j\wedge\omega^{n-j}\\
&\quad-\sum_{i=0}^{d-1}\sum_{j=0}^{n-d-1}a_i\binom{n-d}{j}T^j
\chi^{d-i+j}\wedge\omega^{n-d-j+i}.
\end{aligned}
\end{equation*}
Then
\[
H_{G_{T,n-d}}(\omega)=T^{n-d}H_G(\omega)+\mathcal R_T(\omega).
\]
For $n-p \le d-1$, we have that 
\begin{equation*}
\begin{aligned}
H_G^{(n-p)}(\omega)
&=\frac{d!}{(d-n+p)!}\omega^{d-n+p}\wedge\chi^{n-d}\\
&\quad-\sum_{i=n-p}^{d-1}\frac{i!a_i}{(i-n+p)!}
\omega^{i-n+p}\wedge\chi^{n-i}.
\end{aligned}
\end{equation*}
\begin{equation*}
H_{G_{T,n-d}}^{(n-p)}(\omega)
=T^{n-d}H_G^{(n-p)}(\omega)+\mathcal R_T^{(n-p)}(\omega).
\end{equation*}
Consequently,
\begin{equation}
\begin{aligned}
\int_V H_{G_{T,n-d}}^{(n-p)}(\omega)
&=T^{n-d}\int_VH_G^{(n-p)}(\omega)
+\int_V\mathcal R_T^{(n-p)}(\omega)\\
&\ge\epsilon_0T^{n-d}\int_V\chi^p
-CT^{n-d-1}\int_V\chi^p>0.
\end{aligned}
\end{equation}
In the last line of the above formula, we use the fact that $-C_1 \chi \le \omega \le C_1 \chi$ for some constant $C_1$. We also choose $T$ sufficiently large that $\epsilon_0 T >C$ and $T \ge 1$.

For $n-p = d$, we have that 
\begin{equation*}
H_G^{(n-p)}(\omega)=d!\,\chi^{n-d},
\end{equation*}
where the sum in the general formula is empty because $n-p=d$.
\begin{equation*}
H_{G_{T,n-d}}^{(n-p)}(\omega)
=T^{n-d}d!\,\chi^{n-d}+\mathcal R_T^{(n-p)}(\omega).
\end{equation*}
Following the same calculation as in the case $n-p\le d-1$, we obtain
\begin{equation*}
    \int_V H_{G_{T,n-d}}^{(n-p)}(\omega) >0.
\end{equation*}

For $n-p>d$, we have that:
\begin{equation*}
\begin{split}
H_{G_{T,n-d}}^{(n-p)}(\omega)&= \sum_{j=0}^{n-d-1} \binom{n-d}{j} T^j \chi^{j} \wedge (\omega^{n-j})^{(n-p)} \\
&- \sum_{i=0}^{d-1} \sum_{j=0}^{n-d-1} a_i \binom{n-d}{j}T^j \chi^{d-i+j} \wedge (\omega^{n-d-j+i})^{(n-p)}\\
&= \binom{n-d}{p}T^p (n-p)! \chi^p + I.
\end{split}
\end{equation*}
Here $-C \chi^p T^{p-1} \le I \le C \chi^p T^{p-1}$, because $T \ge 1$ and $-C_1 \chi \le \omega \le C_1 \chi$ for some constant $C_1$. It follows that
\begin{equation*}
\begin{split}
    \int_V H_{G_{T,n-d}}^{(n-p)}(\omega) &= \int_V \binom{n-d}{p}T^p (n-p)! \chi^p + \int_V I \\
& \ge \int_V \binom{n-d}{p}T^p (n-p)! \chi^p - C \int_V T^{p-1}\chi^p>0.
\end{split}
\end{equation*}
In the last inequality above, we choose $T$ sufficiently large that $\binom{n-d}{p}T (n-p)! >C$.
\end{proof}

We are now ready to prove Theorem~\ref{MainThm2} assuming Theorem~\ref{MainThm1}.

\begin{proof}
(Proof of Theorem~\ref{MainThm2}.) $(1) \Longrightarrow (2)$. First, since $r(f)>r(f')$, the polynomial $g(x)=f(x)+\epsilon$ remains right-Noetherian for sufficiently small $\epsilon$ because $f$ and $g$ have the same derivatives of every positive order. For $g$, condition~(\ref{e uniform stable}) also holds for $V=M$, while the integral is unchanged for $V\ne M$. By Lemma~\ref{cone-inclusion}, $g_{T,n-d}$ is right-Noetherian. By Proposition~\ref{prop hgtn-k}, $g_{T,n-d}$ satisfies condition~(\ref{e uniform stable}) for every $V$. For any $\delta>0$, $h_{T,n-d}=g_{T,n-d}-\delta g_{T,n-d}'$ is strongly strictly right-Noetherian by Lemma~\ref{perturbation}. We choose $\delta>0$ sufficiently small that condition~(\ref{e uniform stable}) still holds, also for $V=M$. Let $c>0$ be the normalization constant such that $h_{T,n-d,c}=h_{T,n-d}-c$ satisfies the integrability condition. Note that $r(h_{T,n-d,c})\ge r(h_{T,n-d})$, while these polynomials have the same derivatives of every positive order. Hence $h_{T,n-d,c}$ also satisfies the strongly strictly right-Noetherian condition. Using Theorem~\ref{MainThm1}, we can find a smooth function $u$ satisfying
\begin{equation*}
H_{h_{T,n-d}}(\omega_u)= c \chi^n, \qquad
\lambda[\chi^{-1}\omega_u]\in \Upsilon^1_{h_{T,n-d,c}}=\Upsilon^1_{h_{T,n-d}}.
\end{equation*}
By Lemma~2.5 of \cite{LIN2023110038}, $\omega_u\in \Upsilon_{h_{T,n-d}}$. Lemma~\ref{perturbation}, Lemma~\ref{cone-inclusion}, Theorem~\ref{UpsilonDominance} then show that $\omega_u\in \Upsilon_g$.

$(2) \Longrightarrow (1)$.

By assumption, there exists \(u\in C^\infty(M,\mathbb{R})\)
such that
\[
\widehat\omega
=
\omega+\sqrt{-1}\,\partial\bar\partial u
\]
satisfies
\[
\lambda\!\left[\chi^{-1}\widehat\omega\right](x)
\in \Upsilon^1_F
\qquad
\text{for every }x\in M.
\]
Here $F$ is the polarization of $f$. By Lemma~\ref{positivity}, $H_F^{(n-p)}(\widehat\omega)$ is a smooth strictly positive $(p,p)$-form on $M$. Thus there exists a constant $\epsilon_1>0$ such that
\begin{equation*}
H_F^{(n-p)}(\widehat\omega) \ge \epsilon_1 \chi^p.
\end{equation*}
This implies that for any $p$-dimensional irreducible analytic subvariety $V \subset M$,
\begin{equation*}
\int_V H_F^{(n-p)}(\omega)= \int_V H_F^{(n-p)}(\widehat\omega) \ge \epsilon_1 \int_V \chi^p.
\end{equation*}
\end{proof}

\section{The continuity path}

We now prove Theorem~\ref{MainThm1}.
Theorem 1.4 of Lin's paper \cite{lin2023solvabilitygeneralinversesigmak} proved the equivalence of conditions~(1) and~(2), while Lemma~\ref{positivity} gives (1)$\Rightarrow$(3). It therefore remains to prove (3)$\Rightarrow$(2).

We first make two reductions in equation~(\ref{FirstEquation}).
First, we replace $\omega_0$ by $\omega_0-c_{n-1}\chi$ and rewrite the problem. The roots of the new polynomial and of its derivatives are obtained by shifting the corresponding roots by $c_{n-1}$. Thus the new equation is still strongly strictly right-Noetherian, and we may assume without loss of generality that $c_{n-1}=0$. Lemma~\ref{positivity} with $p=1$ then shows that any $\omega$ satisfying the cone condition is K\"ahler.

Second, we reduce to the case in which $c_0(z)$ is constant on $M$.
When $c_0(z)$ is a function on $M$, consider the equation
\begin{equation}\label{ConstantEquation}
\omega^{n} = \sum_{k=1}^{n-2}c_k\binom{n}{k}\omega^{k}\wedge\chi^{n-k} + c_0'\chi^{n}
\end{equation}
where $c_0' =\frac{\int_M c_0(z)\chi^n}{\int_M\chi^n} $ is a constant.
We claim that (\ref{ConstantEquation}) is still strongly strictly right-Noetherian. Since $c_0(z)$ is continuous on $M$, the mean-value theorem for integrals gives a point $z_0\in M$ such that
$$
c_0' = c_0(z_0).$$
The original equation with coefficient $c_0(z)$ is strongly strictly right-Noetherian at every point. The condition at $z_0$ therefore implies that (\ref{ConstantEquation}) is strongly strictly right-Noetherian.

The cone condition~(2) and the numerical condition~(3) in Theorem~\ref{MainThm1} are the same for the $c_0(z)$ and $c_0'$ equations. Thus a $C$-subsolution obtained for the $c_0'$ equation is also a $C$-subsolution for the $c_0(z)$ equation, proving (3)$\Rightarrow$(2) in the variable-coefficient case.

We may therefore assume that $c_{n-1}=0$ and that $c_0$ is constant on $M$.
We argue by induction on the complex dimension $n$ of $M$.
When $n=1$, equation~(\ref{FirstEquation}) becomes $\omega=c_0\chi$, and the result follows. For $n>1$, assume that (3)$\Rightarrow$(2) whenever $\dim_{\bC}M<n$, and now let $\dim_{\bC}M=n$.

Since $M$ is projective, there exists an ample line bundle $L$ on $M$. Let $\chi_L$ be a K\"{a}hler metric in the cohomology class of $2\pi c_1(L)$.

        We use the following continuity path for $t \ge 0$:
        \begin{equation}
            \omega_t^n = \sum_{k=1}^{n-2} c_k\binom{n}{k}\omega_t^k\wedge \chi^{n-k} + \tilde{c}_0(t)\chi^{n}, \quad \omega_t\in\Upsilon^1,    \label{Path2}
        \end{equation}
        where
        \[\omega_t=\omega_0+t\chi_L+\sqrt{-1}\partial\bar\partial\varphi_t,
        \]
        \[
        \tilde c_0(t)=
        \frac{\int_M\left((\omega_0+t\chi_L)^n-\sum_{k=1}^{n-2}c_k\binom{n}{k}
        (\omega_0+t\chi_L)^k\wedge\chi^{n-k}\right)}{\int_M\chi^n}.
        \]
        For convenience, write $\omega_{t,0}=\omega_0+t\chi_L$.

        We need to check the following three conditions for the path~(\ref{Path2}):
        \begin{enumerate}
            \item The integrability constraint, i.e., 
            \begin{equation}
            \int_M\omega_{t,0}^n = \int_M \Big(\sum_{k=1}^{n-2} c_k\binom{n}{k}\omega_{t,0}^k\wedge \chi^{n-k} + \tilde{c}_0(t)\chi^{n}\Big).
            \label{integrabilitypath}
            \end{equation}
            \item The numerical constraint, i.e., 
            \begin{equation}
            \int_V \omega_{t,0}^{p} - \sum_{k=0}^{p-2}c_{k+n-p} \binom{p}{k}\omega_{t,0}^{k}\wedge\chi^{p-k}>0
            \label{numericalpath}
            \end{equation}
            for every $p$-dimensional subvariety $V$ of $M$ with
            $1\le p\le n$, and for every $t>0$.
            \item The Positivstellensatz constraint, i.e., $(c_{n-2},\ldots,c_{1},\tilde{c}_{0}(t))$ is strongly strictly right-Noetherian for $t>0$.
        \end{enumerate} 

        The integrability condition follows immediately from the definition of
        $\tilde{c}_0(t)$. We first verify the numerical constraint. For
        $1\le q\le n$, set
        \[
        \Phi_q
        :=\omega_0^q-
        \sum_{r=0}^{q-2}c_{r+n-q}\binom{q}{r}
        \omega_0^r\wedge\chi^{q-r},
        \qquad \Phi_0:=1.
        \]
        For $\ell\le k\le p$, using the binomial identity
        \[
        \binom{p}{k}\binom{k}{\ell}
        =\frac{p!}{(p-k)!(k-\ell)!\ell!}
        =\binom{p}{\ell}\binom{p-\ell}{k-\ell},
        \]
        we see that
        \begin{equation}
        \begin{split}
        &\int_V\bigg((\omega_0+t\chi_L)^p
        -\sum_{k=0}^{p-2}c_{k+n-p}\binom{p}{k}
        (\omega_0+t\chi_L)^k\wedge\chi^{p-k}\bigg)\\
        &=\sum_{\ell=0}^{p}t^\ell
        \int_V\chi_L^\ell\wedge
        \bigg(\binom{p}{\ell}\omega_0^{p-\ell}
        -\sum_{k=\ell}^{p-2}c_{k+n-p}
        \binom{p}{k}\binom{k}{\ell}
        \omega_0^{k-\ell}\wedge\chi^{p-k}\bigg)\\
        &=\sum_{\ell=0}^{p}\binom{p}{\ell}t^\ell
        \int_V\chi_L^\ell\wedge
        \bigg(\omega_0^{p-\ell}
        -\sum_{k=\ell}^{p-2}c_{k+n-p}
        \binom{p-\ell}{k-\ell}
        \omega_0^{k-\ell}\wedge\chi^{p-k}\bigg)\\
        &=\sum_{\ell=0}^{p}\binom{p}{\ell}t^\ell
        \int_V\chi_L^\ell\wedge\Phi_{p-\ell}
        \label{IntegrabilityExpansion}
        \end{split}
        \end{equation}

        It remains to check that every coefficient in
        \eqref{IntegrabilityExpansion} is nonnegative and that at least one is
        positive. The coefficient with $\ell=0$ is nonnegative by the
        assumptions of Theorem~\ref{MainThm1}, while the coefficient with
        $\ell=p$ is $\int_V\chi_L^p>0$. If $1\le\ell<p$, choose $m$
        sufficiently large and generic divisors
        $Y_1,\ldots,Y_\ell\in |mL|$ that meet $V$ properly. Since
        $\Phi_{p-\ell}$ is closed and $[Y_i]=m[\chi_L]$, the intersection
        formula gives
        \begin{equation}\label{DivisorIntersectionIdentity}
        \int_V\chi_L^\ell\wedge\Phi_{p-\ell}
        =\frac{1}{m^\ell}
        \int_{V\cdot Y_1\cdots Y_\ell}\Phi_{p-\ell} \ge 0.
        \end{equation}
        Indeed, the proper intersection on the right is an effective cycle of
        pure dimension $p-\ell$, and the assumptions of
        Theorem~\ref{MainThm1} apply to each of its irreducible components.
        Consequently, the right-hand side of
        \eqref{IntegrabilityExpansion} is strictly positive for every $t>0$,
        proving the numerical constraint.

        We now verify the Positivstellensatz constraint. Taking $p=n$ and
        $V=M$ in \eqref{numericalpath}, and then using
        \eqref{integrabilitypath} and the original integrability condition,
        gives
       \begin{equation}
       \label{ConstantCoefficientComparison}
        0<\int_M\bigg(\omega_{t,0}^n
        -\sum_{k=0}^{n-2}c_k\binom{n}{k}
        \omega_{t,0}^k\wedge\chi^{n-k}\bigg)=\big(\tilde c_0(t)-c_0\big)\int_M\chi^n,
        \end{equation}
        for all $t>0$.
        Consider the two monic polynomials
        \begin{align*}
        f(x)&=x^n-\sum_{k=1}^{n-2}c_k\binom{n}{k}x^k-c_0,\\
        f_t(x)&=x^n-\sum_{k=1}^{n-2}c_k\binom{n}{k}x^k-\tilde c_0(t).
        \end{align*}
        By assumption, $f$ is strongly strictly right-Noetherian. Let $r(f)$ be the largest
        real root of $f$. By \eqref{ConstantCoefficientComparison},
        \[
        f_t(r(f))=c_0-\tilde c_0(t)<0,
        \qquad \lim_{x\to+\infty}f_t(x)=+\infty.
        \]
        The intermediate value theorem therefore gives a real root of $f_t$
        strictly larger than $r(f)$. In particular, its largest real root
        satisfies
        \[
        r(f_t)>r(f)>r(f')=r(f_t').
        \]
        The remaining strict inequalities in the root sequence are unchanged,
        because the positive-order derivatives of $f_t$ and $f$ coincide.
        Thus $f_t$ is strongly strictly right-Noetherian, proving the
        Positivstellensatz constraint.

        Consider the following set:
        \[
        I=\left\{t \ge 0: \eqref{Path2} \text{ has a smooth solution}
        \right\}.
        \]
        For sufficiently large $t$, it is easy to see that $\omega_0+t\chi_L$ satisfies the cone condition. Theorem~1.4 of \cite{lin2023solvabilitygeneralinversesigmak} gives $t\in I$, so $I$ is nonempty. Since the cone condition is open, Theorem~1.4 of \cite{lin2023solvabilitygeneralinversesigmak} shows that $I$ is open. It remains to prove that $I$ is closed. If $t_0\in I$, there exists a smooth solution $\omega_{t_0}=\omega_0+t_0\chi_L+\sqrt{-1}\partial\bar\partial\varphi_{t_0}$ satisfying the cone condition at $t_0$. By Proposition~\ref{PRTmatrix}, $\omega_0+t\chi_L+\sqrt{-1}\partial\bar\partial\varphi_{t_0}$ also satisfies the cone condition for all $t>t_0$. Hence $t\in I$ for all such $t>t_0$. Let $t_0=\inf I$. It suffices to prove that $t_0\in I$, and we may assume without loss of generality that $t_0=0$.

    %%%证明continue method 的闭性
    \section{Mass Concentration}    

        In this section, we prove a mass concentration result.
        \begin{thm}
            Let $(M,\chi)$ be a connected projective manifold of complex dimension $n$ with a fixed K\"ahler form $\chi$, and let $[\omega_0]\in H^{1,1}(M;\bR)$. Suppose that for every $t>0$ there exists $\omega_t\in[\omega_{t,0}]\cap\Upsilon^1$. Then for any divisor $Y$, there exist $\beta_Y>0$ and a current $\Theta\in[\omega_0]$ such that
            $\Theta \geq \beta_Y [Y]$ and $\Theta$ satisfies the cone condition $\bar\Upsilon^1$.
            \label{MassConcentration}
        \end{thm}

        We first define the cone condition for currents. Recall that any closed positive $(1,1)$-current can locally be written as $\sqrt{-1}\partial\db$ applied to a real-valued function.
        \begin{defn}
            Suppose $\chi$ is a K\"{a}hler form.
            Let $\Theta$ be a closed, positive $(1,1)$ current. We say that $\Theta$ satisfies the cone condition $\bar\Upsilon^1(\chi)$
            if for any coordinate chart $U$ with $\Theta|_U=\sqrt{-1}\partial\db\varphi_U$, on $U_{\delta}:=\{x\in M:B(x,\delta)\subset U\}$
            we have $\sqrt{-1}\partial\db \varphi_{U,\delta}\in \bar\Upsilon^1(\chi_0)$
            for any K\"{a}hler metric $\chi_0$ on $U$ with constant coefficients satisfying $\chi_0\leq\chi$.
            Here $\varphi_{U,\delta}$ is the mollification of $\varphi_U$, and $\bar\Upsilon^1(\chi)$ is the closure of $\Upsilon^1(\chi)$.
            \label{DefnDegenerateConeCondition}
        \end{defn}

        \begin{proof} (Proof of Theorem~\ref{MassConcentration}.)
            Suppose $M$ is covered by coordinate balls $\{B_j\}_{j=1}^{N}$ and that $Y$ is defined by local functions $g_j$ on $B_j$.
            Let $\{ \theta_j \}$ be a partition of unity subordinate to $\{B_j\}$.
            For $0\le t\ll1$, define
            \[
            \psi_t=\log\!\left(\sum_j\theta_j|g_j|^2+t^2\right),
            \qquad
            \chi_t=\chi+A^{-1}\sqrt{-1}\partial\db\psi_t,
            \]
            where $A\gg1$ will be chosen below. For $t$ sufficiently small, $\chi_t$ is a smooth K\"ahler form on $M$ that concentrates near $Y$
            and $\chi_t > (1-C_1A^{-1})\chi$ for some fixed $C_1$.
            Here $C_1$ need only satisfy $\sqrt{-1}\partial\db\psi_t>-C_1\chi$. By the result of Demailly and P\u{a}un, $C_1$ is independent of $t$.
            
            Consider the following equation:
            \begin{equation}
            \omega^n = \sum_{k=1}^{n-2}c_k\binom{n}{k }\omega^k\wedge\chi^{n-k} + f_t\chi^n \label{MassEquation}
            \end{equation}
            where  $\omega \in [\omega_{t,0}]$ and
            $$
            f_t = \frac{\chi_t^n}{\chi^n} -1 + \tilde{c}_0(t).$$
            We can directly check that
            \[
            \int_M\omega_{t,0}^n
            =\int_M\left(\sum_{k=1}^{n-2}c_k\binom{n}{k}
            \omega_{t,0}^k\wedge\chi^{n-k}\right)+\int_M f_t\chi^n.
            \]

            We also need to check the strongly strictly right-Noetherian condition. Recall that $x^n-\sum_{k=1}^{n-2}c_k\binom{n}{k}x^k-c_0$ is strongly strictly right-Noetherian. By an argument similar to that in the previous section, there exists $\epsilon_0>0$ such that for all $c>c_0-\epsilon_0$, $x^n-\sum_{k=1}^{n-2}c_k\binom{n}{k}x^k-c$ is also strongly strictly right-Noetherian.
            By choosing $A$ sufficiently large, there exists $\epsilon_1>0$ such that for all $0<t<\epsilon_1$, we have $f_t-\tilde{c}_0(t)>-\epsilon_0$, which implies the strongly strictly right-Noetherian condition.
                
            By assumption, there exists $\omega_t\in[\omega_{t,0}]\cap\Upsilon^1$ for all $0<t<\epsilon_1$. Hence $\omega_t$ also satisfies the cone condition for~(\ref{MassEquation}), and we obtain a smooth solution $\omega'_t\in[\omega_{t,0}]\cap\Upsilon^1$ for $0<t<\epsilon_1$; that is,
            $$ 
            (\omega'_t)^n = \sum_{k=1}^{n-2}c_k\binom{n}{k }(\omega'_t)^k\wedge\chi^{n-k} + f_t\chi^n.$$    
            
            Since $\int_M \omega'_t\wedge\chi^{n-1} = \int_M \omega_0\wedge\chi^{n-1} + t\int_M \chi_L\wedge\chi^{n-1}$ is bounded,
            there exists a subsequence $t_i \to 0$ such that $\omega'_{t_i}\to\Theta$, 
            where $\Theta$ is a positive current in the class $[\omega_0]$.

            We use the following lemma, proved below.
            \begin{lem}
                For any neighborhood $U$ of a point $y\in Y$, there exist constants $\delta_U>0$ and $0<t_U<1$
                such that for all $0<t<t_U$ the following inequality holds:
                $$
                \int_{U\cap V_t} \omega'_t\wedge \chi^{n-1}>\delta_U,$$
                where $V_t = \{ \psi_0 < \log t \}$.
                \label{PositiveMass}
            \end{lem}

            Consider the current $1_Y\Theta$. By the Skoda-El Mir extension theorem, $1_Y\Theta$ is a closed, nonnegative $(1,1)$-current supported on $Y$. By the standard support theorem, $1_Y\Theta=\sum_i\beta_i[Y_i]$ for some $\beta_i\geq0$, where the $Y_i$ are the irreducible components of $Y$. Lemma~\ref{PositiveMass} gives $\beta_i>0$ for every $i$.
            Hence we can take $\beta = \min\beta_i$, and $\Theta \geq \beta [Y]$.
        
            We next show that $\Theta$ satisfies the cone condition.
            For each $t$, we have that 
            $\omega'_t \in \Upsilon^1(\chi)$.
            Let $U$ be a coordinate chart, and write $\Theta |_U = \sqrt{-1}\partial\db \varphi$ and $\omega'_t |_U = \sqrt{-1}\partial\db \varphi_t$.
            Let $\varphi_{\delta},\varphi_{t,\delta}$ be the convolutions of $\varphi$ and $\varphi_t$ with the standard mollifier.
            Let $\chi_0$ be a constant-coefficient form in $U$ such that $\chi\geq\chi_0$ on $B_{\delta}(x)$. Since $\omega'_t\in \Upsilon^1(\chi)$, Proposition~\ref{PRTmatrix} gives $\omega'_t\in \Upsilon^1(\chi_0)$.

            By the convexity of the $\Upsilon$-cone and Corollary \ref{convexmatrix}, $\sqrt{-1}\partial\db \varphi_{t,\delta}\in \Upsilon^1(\chi_0)$. Hence $\Theta$ satisfies the cone condition $\bar\Upsilon^1$ for currents.

        \end{proof}

        We now prove Lemma~\ref{PositiveMass}.
        \begin{proof}
            We consider the form $\omega'_t + \epsilon_2^{-1}\chi$ where $\epsilon_2>0$ will be determined later.
            Since $\omega'_t$ solves the equation
            $$
            (\omega'_t)^n = \sum_{k=1}^{n-2}c_k\binom{n}{k}(\omega'_t)^{k}\wedge\chi^{n-k}+f_t\chi^n.$$
            we have
            $$
            (\omega'_t+\epsilon_2^{-1}\chi)^n = \sum_{k=1}^{n-1}(c_k+\epsilon_2^{-n+k})\binom{n}{k}(\omega'_t)^{k}\wedge\chi^{n-k}+\epsilon_2^{-n}\chi^{n}+f_t\chi^n.$$
            Recall that $c_{n-1}=0$. Choose $\epsilon_2>0$ such that for all $0<t\leq \epsilon_1$ and $1\leq k\leq n-1$, 
            we have 
            $$c_k + \epsilon_2^{-n+k} > 0,\epsilon_2^{-n} + \tilde{c}_0(t) -1 >0 .$$
            
            Since $|\tilde{c}_0(t)|$ has a uniform upper bound, we can choose $\epsilon_2>0$ independent of $t$.
            Then
            $$
            (\omega'_t+\epsilon_2^{-1}\chi)^n > (f_t+\epsilon_2^{-n})\chi^n > \chi_t^n.$$

            Let $\lambda_1(z)\geq \cdots \geq \lambda_n(z)\geq 0 $ be the eigenvalues of $\omega'_t+\epsilon_2^{-1}\chi$ with respect to $\chi_t$.
            Since $(\omega'_t+\epsilon_2^{-1}\chi)^{n-1}\wedge \chi_t \geq \lambda_1\cdots\lambda_{n-1}\chi_t^n$, we have
            $$
            \int_M \lambda_1\cdots\lambda_{n-1}\chi_t^n \leq \int_M (\omega'_t+\epsilon_2^{-1}\chi)^{n-1}\wedge \chi_t  \leq C_2 $$
            where $C_2$ is a constant depending on $\epsilon_2,[\omega_0],[\chi]$.
            
            For any $\delta>0$, if we set $E_{\delta} = \{ z\in M: \lambda_1\cdots\lambda_{n-1} > C_2\delta^{-1} \}$, then
            $$
            \int_{E_{\delta}} \chi_t^n = \int_{E_{\delta}} \frac{C_2\delta^{-1}}{C_2\delta^{-1}}\chi_t^n < \frac{\int_{E_{\delta}} \lambda_1\cdots\lambda_{n-1}\chi_t^n}{C_2\delta^{-1}} \leq \delta.$$
            We have $\chi_t > (1-C_1A^{-1})\chi$ for some fixed $C_1$ independent of $t$. Choose $A\gg1$ such that $\chi_t > (1-C_1A^{-1})\chi >\frac{1}{2}\chi$. Then
            $$
            \int_{E_{\delta}} \chi_t\wedge\chi^{n-1} < 2^{n-1}\int_{E_{\delta}}\chi_t^{n} < 2^{n-1}\delta$$
            By Lemma~2.1 of \cite{JPDemaillyPaun}, there exists $\delta(U)>0$ such that for all $t\ll1$,
            \begin{align*}
            \int_{U\bigcap V_t \bigcap E_{\delta}^c}\chi_t\wedge\chi^{n-1} 
            &= 
            \int_{U\bigcap V_t }\chi_t\wedge\chi^{n-1}-\int_{U\bigcap V_t \bigcap E_{\delta}}\chi_t\wedge\chi^{n-1}
            \\
            &\geq
            \delta(U) - 2^{n-1}\delta
            \\
            &=
            \frac{9}{10}\delta(U)            
            \end{align*}
            where we let $\delta := \frac{1}{10\cdot2^{n-1}}\delta(U)$.
            Here we assume that $t<\epsilon_3$ for some constant $0<\epsilon_3<\epsilon_1$.

            By the choice of $\epsilon_2$, we have $(\omega'_t+\epsilon_2^{-1}\chi)^n > \chi_t^n $.
            For $z\in E_{\delta}^{c}$ we have 
            $$
            (\omega'_t + \epsilon_2^{-1}\chi)(z) \geq \frac{\delta}{C_2}   \chi_t(z)$$

            Consequently,
            \begin{align*}
                \int_{U\cap V_t} \omega'_t\wedge\chi^{n-1} 
                \geq& \int_{U\cap V_t\cap E_{\delta}^{c}}\omega'_t\wedge\chi^{n-1} \\
                =& \int_{U\cap V_t\cap E_{\delta}^{c}}(\omega'_t + \epsilon_2^{-1}\chi)\wedge\chi^{n-1} - \epsilon_2^{-1}\chi^{n}\\
                \geq& \int_{U\cap V_t\cap E_{\delta}^{c}}\frac{\delta}{C_2} \chi_t\wedge\chi^{n-1} - \epsilon_2^{-1}\chi^{n}\\
                \geq& \frac{\delta}{C_2}\frac{9}{10}\delta(U)  - \int_{U\cap V_t\cap E_{\delta}^{c}}\epsilon_2^{-1}\chi^{n}
            \end{align*}

            Since $\int_{U\cap V_t\cap E_{\delta}^{c}}\chi^{n} \leq \int_{U\cap V_t}\chi^{n}$, 
            we have
            $$
                \int_{U\cap V_t} \omega'_t\wedge\chi^{n-1} 
                \geq \frac{\delta}{C_2} \frac{9}{10}\delta(U) -  \epsilon_2^{-1}\int_{U\cap V_t}\chi^{n}
            $$

            Hence, if we choose $t<t_U$ for some $t_U<\epsilon_3$ depending on $\chi$, $Y$, $\epsilon_2$, and $C_2$,
            then 
            $\int_{U\cap V_t} \chi^{n} <\frac{1}{10} \frac{\delta}{C_2\epsilon_2^{-1}}\delta(U). $
            It follows that $\int_{U\cap V_t}\omega'_t\wedge\chi^{n-1}>\delta_U$ for $\delta_U=\frac{4\delta\delta(U)}{5C_2}$.

        \end{proof}

     \section{Main proof}

        As a corollary of the previous section, we can now apply the following gluing theorem proved by Datar-Pingali \cite{Datar2020ANC}, which relaxes the strict cone condition on $M$ in \cite{GaoChen2021} to the strict cone condition on $M\backslash Y$.
        
        \begin{thm}[Proposition 4.1 of \cite{Datar2020ANC}]
            Let $(M,\chi)$ be a compact K\"{a}hler manifold, let $[\omega_0]$ be a nef class, and let $L$ be an ample line bundle. For sufficiently large $N$, choose a smooth divisor $Y \in |N L|$ and a K\"{a}hler metric $\chi_{Y}=N\chi_L$ in the cohomology class of $Y$. Let $0< \beta <1$ and $c_n$ be a dimensional constant.
            Assume $T\geq \beta [Y]$ is a positive current in $[\omega_0-\varepsilon\chi_Y]$ satisfying the strict cone condition on $M\backslash Y$ 
            where $\varepsilon$ satisfies the upper bound $\varepsilon < \frac{\beta}{1000c_n}$.
            Then as long as for any analytic subvariety $Z$, there exists a neighborhood $U$ of $Z$ with a K\"{a}hler form $\omega = \omega_0 |_U + \sqrt{-1}\partial\db\psi_U$ 
            satisfying the cone condition on $U$,
            then there exists a K\"{a}hler form $\hat{\omega} = \omega_0 +  \sqrt{-1}\partial\db\hat{\psi}$ satisfying the cone condition.
        \label{GluingThm}
        \end{thm}

        \begin{rmk}
            Here the strict cone condition means that the inequality in Definition~\ref{DefnDegenerateConeCondition} is strict.
        \end{rmk}

        The proof in \cite{Datar2020ANC} assumes that $[\omega_0]$ is a K\"ahler class, whereas here we know only that it is nef. Thus, we must replace the term $\varepsilon\omega_0$ in their proof by the K\"ahler form $\varepsilon\chi_Y$. The rest of the proof applies without essential changes because it uses only the convexity of the cone condition and the fact that the cone condition implies the K\"ahler condition, allowing pluripotential theory to be applied.

        Assuming that the extension theorem required by Theorem~\ref{GluingThm} holds, we prove Theorem~\ref{MainThm1} along path~(\ref{Path2}).
        Recall that we use the continuity set
        \[
        I=\left\{t \ge 0: \eqref{Path2} \text{ has a smooth solution}
        \right\}.
        \]
        It remains to prove closedness. Let $t_0=\inf I$; we may assume that $t_0=0$ and need only prove that $0\in I$.

        By Theorem~\ref{MassConcentration}, there exist $\beta_Y>0$ and a current $\Theta\in[\omega_0]$ such that $\Theta\geq\beta_Y[Y]$. Moreover, $\Theta$ satisfies the cone condition $\bar\Upsilon^1$.

        Consider the form $\Xi = \Theta - \frac{\beta_{Y}}{2}[Y] +  (\frac{\beta_{Y}}{2}-\epsilon_4)\chi_{Y}\in[\omega_0-\epsilon_4\chi_{Y}]$ for $\epsilon_4<\frac{\beta_{Y}}{2}$. By Lemma~\ref{strictPRT} and Proposition~\ref{PRTmatrix}, $\Xi$ satisfies the strict cone condition on $M\backslash Y$, and
        $\Xi \geq \frac{\beta_Y}{2} [Y]$.
        Moreover, we can choose $\epsilon_4$ smaller such that $\epsilon_4 < \frac{\beta}{1000c_n}$.
        
        This completes the proof, assuming the existence of such $U$ and $\omega_U$, which is proved in the next section.
       
   \section{Extension theorems}

        First, we consider the smooth case:
        \begin{lem}
            Let $Z$ be a smooth $m$-dimensional subvariety of $M$. Suppose there exists a K\"ahler metric $\omega_Z = \omega_0 |_Z + \sqrt{-1}\partial\db\psi_Z$
            on $Z$ such that $\omega_Z, \chi_Z$ is in the $\Upsilon$ cone of 
            \(
            x^m - \sum_{j=0}^{m-2}c_{n-m+j}\binom{m}{j}x^j.\)
            
            Then there exists a neighborhood $U$ of $Z$ and a smooth function $\psi_U:U\to\bR$ such that the smooth form $\omega_U = \omega_0+\sqrt{-1}\partial\db\psi_U$
            is K\"{a}hler and is in the $\Upsilon^1$ cone of
            \(
            x^{n} - \sum_{k=0}^{n-2}c_k\binom{n}{k}x^k.\)
        \label{SmoothExtention}
        \end{lem}

        \begin{proof}
            %The proof may use the $\Upsilon_k$-cone.

            Let $(U_{\alpha},\{z_{\alpha}^i\})$ be finitely many coordinate charts in $M$ that cover $Z$, with
            $$
            Z\cap U_{\alpha}=\{z_{\alpha}^{m+1}=\cdots=z_{\alpha}^{n}=0\}.$$
            Let $\rho_{\alpha}(z_{\alpha}^{1} , \cdots, z_{\alpha}^{m})$ be a partition of unity on $Z$ subordinate to $U_{\alpha}\cap Z$.
            Let $\pi^{\alpha}$ be the projection to the first $m$ coordinates.
            Let $\eta_{\alpha}(z_{\alpha}^{m+1} , \cdots)$ be a smooth function with compact support in the vertical part of $U_{\alpha}$
             such that it is identically 1 in a neighbourhood of the origin.
            Then 
            \[
            \tilde{\rho}_{\alpha}(z_{\alpha})=
            \frac{\eta_{\alpha}(z_{\alpha}^{m+1},\ldots,z_{\alpha}^{n})
            (\pi^{\alpha})^*\rho_{\alpha}}
            {\sum_\beta\eta_{\beta}(z_{\beta}^{m+1},\ldots,z_{\beta}^{n})
            (\pi^{\beta})^*\rho_{\beta}}
            \]
            is a partition of unity on a closed subset of $\cup_{\alpha}U_{\alpha}$. 
            
            We  consider the form
            \begin{align*}
                \tilde{\omega}_U &= \omega_0+ \sqrt{-1}\partial\db \Big( \sum_{\alpha}\tilde{\rho}_{\alpha}(\pi^{\alpha})^*\psi_Z \Big)\\
                \omega_U &= \tilde{\omega}_U + 10C\sqrt{-1}\partial\db d_{\chi}(\cdot,Z)^2
            \end{align*}
            where $C$ is a sufficiently large constant. The squared distance $d_{\chi}(\cdot,Z)^2$ is smooth on a sufficiently small tubular neighborhood of $Z$. On such a neighborhood $U$, the form $\omega_U$ is K\"ahler, and by continuity it is enough to verify the cone condition on $Z$. Note that $\omega_U|_Z=\omega_Z$. If we choose $C$ sufficiently large to control the $dz^i\wedge d\bar z^j$ terms of $\tilde\omega_U$ for $i>m$ or $j>m$, then at a point $p\in Z$, after diagonalizing \[\omega_Z=\sum_{i=1}^{m} \lambda_i \sqrt{-1} dz_i\wedge d\bar z_i, \quad \chi=\sum_{i=1}^{n} \sqrt{-1} dz_i\wedge d\bar z_i,\] we have \[
            \omega_U \ge \sum_{i=1}^{m} (\lambda_i-\frac{1}{C}) \sqrt{-1} dz_i\wedge d\bar z_i+C\sum_{i=m+1}^{n} \sqrt{-1} dz_i\wedge d\bar z_i.
            \]

            Since $Z$ is compact, and $\omega_Z$ is in the open $\Upsilon$ cone, we can use Lemma~\ref{extensionlemma} to finish the proof.

        \end{proof}

        Now we provide a proof assuming that $Z$ is smooth.

             \begin{thm}
            Let $Z\subset M$ be a smooth subvariety of dimension $m<n$.
            Under condition~(3) of Theorem~\ref{MainThm1}, there exists a K\"{a}hler metric
            $\omega_Z = \omega_0|_Z+\sqrt{-1}\partial\db\psi_Z$ on $Z$  such that
            $(\omega_Z,\chi_Z)$ lies in the $\Upsilon$-cone associated with
            \(
            x^m - \sum_{j=0}^{m-2}c_{n-m+j}\binom{m}{j}x^j.\)
        \label{InductionSmooth}
        \end{thm}

        \begin{proof}
            Let $c_Z$ be the constant determined by
            $$
            c_Z\int_Z\chi^m = \int_Z \omega_0^m - \sum_{j=0}^{m-2}c_{n-m+j}{m\choose j}\int_Z \omega_0^j\wedge\chi^{m-j}.$$
            The integral condition implies that $c_Z>0$.
            Consider the equation
            $$
            \omega^m = \sum_{j=0}^{m-2}c_{n-m+j}{m\choose j} \omega^j\wedge\chi^{m-j} + c_Z\chi^m$$
            for $\omega\in [\omega_0|_Z]$.
            By assumption, $c_Z>0$, which implies the strong strict $\Upsilon$-stability of the equation on $Z$.
            The integral condition also gives, for every subvariety $V\subset Z$ of dimension $k$,
            $$
            \int_V  \omega_0^k - \sum_{j=0}^{k-2}c_{j+n-k}{k\choose j}\int_V \omega_0^j\wedge\chi^{k-j}>0.$$
            By induction, there exists a smooth K\"{a}hler metric $\omega_Z \in [\omega_0|_Z]$ on $Z$ such that
            $$
            \omega_Z^m = \sum_{j=0}^{m-2}c_{n-m+j}{m\choose j}\omega_Z^j\wedge\chi^{m-j} + c_Z\chi^m>\sum_{j=0}^{m-2}c_{n-m+j}{m\choose j}\omega_Z^j\wedge\chi^{m-j},$$
            and $\omega_Z$ is in the $\Upsilon^1$ cone.
            By Lemma~2.5 of \cite{LIN2023110038}, $\omega_Z$ is also in the $\Upsilon$ cone.
        \end{proof}

   We now prove the singular case.

        \begin{lem}
            Let $Z\subset M$ be a subvariety. There exists a neighborhood $U$ of $Z$ and a K\"ahler metric $\omega_U = \omega_0+\sqrt{-1}\partial\db\psi_U$ satisfying the cone condition.
            \label{SingularExtention}
        \end{lem}

        %\begin{rmk}
        %    We prove this lemma by induction on the dim of $M$.
        %    That is to say, assume that for any lower dimension subvarieties which satisfy the intergral condition, 
        %    then there exists the neighborhood.
        %    And now we want to prove the higher dimension case.
        %\end{rmk}

        \begin{proof}

            Without loss of generality, we can assume that $Z$ is connected.
            We induct on the maximal dimension $m<n$ of the irreducible components of $Z$. For $m=0$, $Z$ is a point, and the lemma is immediate.
            Assume it is true for $0,1,\cdots,m-1$ and removing the isolated points of $Z$,
            we may assume that the dimension of  every irreducible component of $Z$ is $\geq 1$.

            \begin{itemize}
                \item If $Z$ is smooth, then we are done by Theorem \ref{InductionSmooth} and Lemma~\ref{SmoothExtention}.
                
                \item If $Z$ is singular, we use a canonical resolution of singularities.
                There exists a morphism $\pi\colon\tilde{M}_r\to\tilde{M}_{r-1}\to\cdots\to\tilde{M}_0=M$ obtained by blowups
                along smooth centres such that the proper transform $\tilde{Z}_r$ of $Z$ is smooth. We only need to assume that $r=1$.
                So $\pi: \tilde{M} \to M$, $\tilde{Z}$ the proper transform of $Z$.
                Let $\tilde{E}$ be the exceptional divisor.

                If $h$ is a Hermitian metric on $[\tilde{E}]$,
                then for some $C_4\gg1$,
                $\tilde{\rho} = \pi^*\chi_Y + C_4^{-1}\sqrt{-1}\partial\db\log h$ is a K\"{a}hler metric on $\tilde{M}$.
                We define
                $$
                    \begin{cases}
                        \tilde{\omega}_{s} = (1+s)\pi^*\omega_0 + s\tilde{\rho}  \\
                        \tilde{\chi}_{s} = \pi^*\chi + s^{n}\tilde{\rho} 
                    \end{cases}
                $$

                We next verify the numerical condition on $\tilde Z$.  The
                organization of the following blow-up expansion is the same as
                the cohomological device used in the proof of
                Lemma~4.5 of~\cite{Datar2020ANC}.

                Let $\tilde V\subseteq\tilde Z$ be an irreducible
                subvariety of dimension $p$. We claim that, for all sufficiently
                small $s>0$,
                \begin{equation}
                \mathcal I_{\tilde V}(s):=
                \int_{\tilde V}\left(
                \tilde\omega_s^p-
                \sum_{j=0}^{p-2}c_{j+n-p}\binom pj
                \tilde\omega_s^j\wedge\tilde\chi_s^{p-j}
                \right)>0.
                \label{Equation2}
                \end{equation}
                More importantly, the upper bound on $s$ can be chosen
                independently of $\tilde V$. If $\tilde V=\tilde Z$, we get the inequality as long as $s$ is small. The case $p=0$ is also immediate, so we assume $1\leq p\leq m-1$.

                Put $V=\pi(\tilde V)$.  Since $\dim V\leq p<m$, the induction
                hypothesis gives a neighborhood $U_V$ of $V$ and a K\"ahler form
                \[
                \omega_V=\omega_0+\sqrt{-1}\partial\bar\partial\varphi_V
                \quad\text{on }U_V
                \]
                satisfying the cone condition.  The positive-ray property of the
                cone shows that $(1+s)\omega_V$ also satisfies the cone condition.
                Hence Lemma~\ref{positivity} gives, for $2\leq i\leq p$,
                \begin{equation}
                \Phi_i\big((1+s)\omega_V,\chi\big)>0,
                \qquad
                \Phi_i(\alpha,\beta):=
                \alpha^i-
                \sum_{a=0}^{i-2}c_{a+n-i}\binom ia
                \alpha^a\wedge\beta^{i-a}.
                \label{Equation3}
                \end{equation}
                We also set $\Phi_0=1$ and $\Phi_1(\alpha,\beta)=\alpha$.

                Set
                \(A_s=(1+s)\pi^*\omega_0\).
                Since integration of closed forms over the cycle $[\tilde V]$
                depends only on their cohomology classes, for every $0\leq i\leq p$
                we have
                \begin{align}
                &\int_{\tilde V}\Phi_i(A_s,\pi^*\chi)\wedge\tilde\rho^{p-i}=
                \int_{\tilde V}
                \Phi_i\big((1+s)\pi^*\omega_V,\pi^*\chi\big)
                \wedge\tilde\rho^{p-i}\geq0
                \label{CohomologicalReplacement}
                \end{align}
                by pulling back~\eqref{Equation3} and wedging with the positive form
                $\tilde\rho^{p-i}$.

                We now expand
                \[
                \tilde\omega_s=A_s+s\tilde\rho,
                \qquad
                \tilde\chi_s=\pi^*\chi+s^n\tilde\rho.
                \]
                In a monomial from
                $(A_s+s\tilde\rho)^j(\pi^*\chi+s^n\tilde\rho)^{p-j}$, choose $a$ factors of
                $A_s$ and $b$ factors of $\pi^*\chi$.  The monomial is then
                \[
                \binom ja\binom{p-j}{b}
                s^{j-a+n(p-j-b)}
                A_s^a\wedge \pi^*\chi^b\wedge\tilde\rho^{p-a-b}.
                \]
                The terms with no $s^n\tilde\rho$ factor are precisely those for which
                $b=p-j$.  Grouping these terms according to
                $i=a+b=a+p-j$ and using
                \[
                \binom pj\binom ja
                =\binom pi\binom ia,
                \qquad j=p-i+a,
                \]
                gives the exact identity
                \begin{equation}
                \mathcal I_{\tilde V}(s)
                =\sum_{i=0}^{p}\binom pi s^{p-i}
                \int_{\tilde V}\Phi_i(A_s,\pi^*\chi)\wedge\tilde\rho^{p-i}
                +\mathcal E_{\tilde V}(s),
                \label{BlowupExpansion}
                \end{equation}
                where the terms containing at least one $s^n\tilde\rho$ factor are
                \begin{align}
                \mathcal E_{\tilde V}(s)
                ={}&-\sum_{j=0}^{p-2}c_{j+n-p}\binom pj
                \sum_{a=0}^{j}\sum_{b=0}^{p-j-1}
                \binom ja\binom{p-j}{b}
                s^{j-a+n(p-j-b)}\notag\\
                &\hspace{5em}\cdot
                \int_{\tilde V}A_s^a\wedge \pi^*\chi^b\wedge\tilde\rho^{p-a-b}.
                \label{ErrorTerm}
                \end{align}

                Because $\tilde Z$ is compact, there is a constant $C_5>0$
                such that
                \[
                \pi^*\omega_0\leq C_5\tilde\rho,
                \qquad \pi^*\chi\leq C_5\tilde\rho.
                \]
                Thus, for $0<s\leq1$, $A_s\leq2C_5\tilde\rho$ and $\pi^*\chi\leq C_5\tilde\rho$.
                Moreover, every exponent of $s$ in~\eqref{ErrorTerm} is at least
                $n$, because $p-j-b\geq1$.  It follows that there is a constant
                $K_p>0$, depending only on $p$, $C_5$, and the coefficients $c_k$,
                but not on $\tilde V$, such that
                \begin{equation}
                \left|\mathcal E_{\tilde V}(s)\right|
                \leq K_p s^n\int_{\tilde V}\tilde\rho^p.
                \label{UniformErrorEstimate}
                \end{equation}
                By~\eqref{CohomologicalReplacement}, every summand in
                \eqref{BlowupExpansion} is nonnegative, and its $i=0$ summand is
                $s^p\int_{\tilde V}\tilde\rho^p$.  Consequently,
                \[
                \mathcal I_{\tilde V}(s)
                \geq
                \bigl(s^p-K_ps^n\bigr)\int_{\tilde V}\tilde\rho^p>0
                \]
                whenever $0<s<s_p$ for a sufficiently small $s_p>0$.  Since
                $p\leq m-1<n$ and $K_p$ is independent of $\tilde V$, taking the
                minimum over $1\leq p\leq m-1$ proves~\eqref{Equation2}
                simultaneously for every proper subvariety of $\tilde Z$.

                Next, note that there exists an ample line bundle $L$ over $M$.
                For some large $N$, the bundle $L^N$ has a holomorphic section $S_Y$ whose zero locus $Y$ is a smooth connected hypersurface
                such that $Y_1 = Y \cap Z$ is a divisor in $Z$.
                Define $\tilde{Y} = \pi^{-1}Y \cap \tilde{Z}$,
                $\hat{E} = \tilde{E} \cap \tilde{Z}$.
                Note that $\tilde{Y}$ and $\hat{E}$ are divisors in $\tilde{Z}$.
                We again need a mass-concentration result on $\tilde{Z}$ that concentrates mass on $\tilde{Y}$.
                
                \begin{lem}
                    There exists a positive current $\tilde{\Theta}_s \in [\tilde{\omega}_{s}]$ such that
                    $\tilde{\Theta}_s \geq \beta_{\tilde{Y}}[\tilde{Y}]$ for some constant $\beta_{\tilde{Y}}$ independent of $s$.
                    Moreover, $\tilde{\Theta}_s$ satisfies the cone condition in the current sense with respect to $\tilde{\chi}_s$.
                    \label{SecondConcentrate}
                \end{lem}

                The proof of this lemma is similar to the proof of Theorem~\ref{MassConcentration}. Let
                $$
                \chi_t = \chi +A^{-1}\sqrt{-1}\partial\db \log(\sum_{j}\theta_j\sum_{k}|g_{j,k}|^2+t^2)$$
                be the form on $M$, which is then pulled back and restricted to $\tilde Z$.
                Consider the following equations on $\tilde{Z}$, for $\Omega_s \in [\tilde{\omega}_s]$,
                $$
                \Omega_s^m = \sum_{k=0}^{m-2}c_{n-m+k}\binom{m}{k}\Omega_s^k\wedge\tilde{\chi}_s^{m-k} + d_0\tilde{\chi}_s^m + f_t\tilde{\chi}_s^m,$$
                where 
                $$
                d_0 = \frac{\int_{\tilde{Z}}\Omega_s^m - \sum_{k=0}^{m-2}c_{n-m+k}\binom{m}{k}\Omega_s^k\wedge\tilde{\chi}_s^{m-k}}{\int_{\tilde{Z}}\tilde{\chi}_s^m}$$
                and
                $$
                f_t = \frac{(\pi^*\chi_t |_{\tilde{Z}})^m - (\pi^*\chi |_{\tilde{Z}})^m}{(\tilde{\chi}_s |_{\tilde{Z}})^m}.$$

                The (\ref{Equation2}) shows that $d_0>0$.
                %And because of $\tilde{\chi}_s >\tilde{\chi}_0 = \pi^*\chi$, 
                %there exists a constant $t_1$ independent of $s$ such that
                %for any $t<t_1$, $(c_{n-2},\cdots,c_{n-m}+f_t)$ is strongly strictly $\Upsilon$-stable and so is $(c_{n-2},\cdots,c_{n-m}+f_t+d_0)$.
                Using $\tilde{\chi}_s>\pi^*\chi$ and $\chi_t \geq (1-A^{-1}C_6)\chi$, we can ensure, by choosing $A$ sufficiently large, that $f_t$ is bounded below by the gap needed for $(c_{n-2},\ldots,c_{n-m}+f_t+d_0)$ to be strongly strictly right-Noetherian.

                %证明positive mass
                Since the numerical condition holds on $\tilde{Z}$, 
                there exists a smooth solution $\tilde{\omega}_{s,t}\in [\tilde{\omega}_{s}]$ to the equations.
                Fix $s$, and denote the weak limit of $\tilde{\omega}_{s,t}$ as $t\to 0$ by $\tilde{\Theta}_{s}$.
                We want to show that $\tilde{\Theta}_s$ has positive mass independent of $s$ on $\tilde{Y}$.

                As in Lemma~\ref{PositiveMass}, we prove that
                for any neighborhood $\tilde{U}\subset \tilde{Z}$ of a point $y\in\tilde{Y}$, 
                there exist constants $\delta(\tilde{U})>0$ and $0<t_{\tilde{U}}<t_1$
                such that for all $t<t_{\tilde{U}}$, we have 
                $$
                \int_{\tilde{U}\cap \tilde{V}_t}\tilde{\omega}_{s,t}\wedge(\pi^*\chi)^{m-1} > \delta(\tilde{U}).$$

                The proof proceeds as in Lemma~\ref{PositiveMass}.
                First, there exists a sufficiently small constant $\epsilon_5>0$ depending only on the coefficients $c_i$ such that
                $$(\tilde{\omega}_{s,t} + \epsilon_5^{-1}\tilde{\chi}_s)^m >f_t\tilde{\chi}_s^m + (\pi^*\chi)^m \geq (\pi^*\chi_t)^m.$$

                The set $\{ \pi^*\chi_t \text{ is degenerate }\}$ will not affect the integral result,
                so we only work on the subset $\tilde{Z}_1 = \{  \pi^*\chi_t \text{ is nondegenerate }  \}$.
                Let $\lambda_1(z)\geq \cdots \geq \lambda_m(z)\geq 0 $ be the eigenvalues of $\tilde{\omega}_{s,t}+\epsilon_5^{-1}\tilde{\chi}_s$ with respect to $\pi^*\chi_t$.
                Since $(\tilde{\omega}_{s,t}+\epsilon_5^{-1}\tilde{\chi}_s)^{m-1}\wedge \pi^*\chi_t \geq \lambda_1\cdots\lambda_{m-1}(\pi^*\chi_t)^m$, we have
                $$
                \int_{\tilde{Z}} \lambda_1\cdots\lambda_{m-1}(\pi^*\chi_t)^m \leq \int_{\tilde{Z}} (\tilde{\omega}_{s,t}+\epsilon_5^{-1}\tilde{\chi}_s)^{m-1}\wedge \pi^*\chi_t  \leq C_7 $$
                where $C_7$ is a constant depending on $\epsilon_5$ and $s_2$.
                
                For any $\delta>0$, if we set $E_{\delta} = \{ z\in \tilde{Z}_1: \lambda_1\cdots\lambda_{m-1} > C_7\delta^{-1} \}$, then
                $$
                \int_{E_{\delta}} (\pi^*\chi_t)^m  \leq \frac{\int_{E_{\delta}} \lambda_1\cdots\lambda_{m-1}(\pi^*\chi_t)^m}{C_7\delta^{-1}} \leq \delta.$$
                By choosing small $\varepsilon_1$, we have $2\pi^*\chi_t > \pi^*\chi$.
                Moreover,
                $$
                \int_{E_{\delta}} \pi^*\chi_t \wedge(\pi^*\chi)^{m-1} < 2^{m-1}\int_{E_{\delta}}(\pi^*\chi_t)^{m} < 2^{m-1}\delta.$$
                By Lemma~2.1 of \cite{JPDemaillyPaun}, there exists $\delta(\tilde{U})>0$ such that for all $t<t_{\tilde{U}}$, the set $U=\pi(\tilde{U})$ is a neighborhood of $Y$ in $Z$, and
                $$
                \int_{U\cap V_t}\chi_t\wedge\chi^{m-1}\geq \delta(\tilde{U}),$$
                where $V_t = \{ \log(\sum_{j}\theta_j\sum_{k}|g_{j,k}|^2) < \log t \}$.

                Let $\tilde{V}_t=\pi^{-1}V_t\cap\tilde{Z}$. Then
                \begin{align*}
                &\int_{\tilde U\cap\tilde V_t\cap E_\delta^c}
                \pi^*\chi_t\wedge(\pi^*\chi)^{m-1}\\
                &=\int_{\tilde U\cap\tilde V_t}
                \pi^*\chi_t\wedge(\pi^*\chi)^{m-1}\\
                &\quad-\int_{\tilde U\cap\tilde V_t\cap E_\delta}
                \pi^*\chi_t\wedge(\pi^*\chi)^{m-1}\\
                &\geq\delta(\tilde U)-2^{m-1}\delta
                =\frac{9}{10}\delta(\tilde U).
                \end{align*}
                where we let $\delta := \frac{1}{10\cdot 2^{m-1}}\delta(\tilde{U})$.

                By the choice of $\epsilon_5$, 
                we have $(\tilde{\omega}_{s,t} + \epsilon_5^{-1}\tilde{\chi}_s)^m > (\pi^*\chi_t)^m$.
                For $z\in E_{\delta}^{c}$ we have 
                $$
                \tilde{\omega}_{s,t}(z) + \epsilon_5^{-1}\tilde{\chi}_s(z) \geq \frac{\delta}{C_7}  \pi^*\chi_t(z)$$
                Then 
                \begin{align*}
                    &\int_{\tilde{U}\cap \tilde{V}_t} \tilde{\omega}_{s,t}\wedge(\pi^*\chi)^{m-1} \\
                    \geq& \int_{\tilde{U}\cap \tilde{V}_t\cap E_{\delta}^{c}}\tilde{\omega}_{s,t}\wedge(\pi^*\chi)^{m-1} \\
                    =& \int_{\tilde{U}\cap \tilde{V}_t\cap E_{\delta}^{c}}(\tilde{\omega}_{s,t} + \epsilon_5^{-1}\tilde{\chi}_s)\wedge(\pi^*\chi)^{m-1} - \epsilon_5^{-1}\tilde{\chi}_s\wedge(\pi^*\chi)^{m-1}\\
                    \geq& \int_{\tilde{U}\cap \tilde{V}_t\cap E_{\delta}^{c}}\frac{\delta}{C_7}  \pi^*\chi_t \wedge(\pi^*\chi)^{m-1} - 
                            \epsilon_5^{-1}\tilde{\chi}_s\wedge(\pi^*\chi)^{m-1}\\
                    \geq& \frac{\delta}{C_7}  \frac{9}{10}\delta(\tilde{U}) -
                        \int_{\tilde{U}\cap \tilde{V}_t\cap E_{\delta}^{c}}\epsilon_5^{-1}\tilde{\chi}_s\wedge(\pi^*\chi)^{m-1}
                \end{align*}

                Since 
                $$
                \int_{\tilde{U}\cap \tilde{V}_t\cap E_{\delta}^{c}}\tilde{\chi}_s\wedge(\pi^*\chi)^{m-1}< \int_{\tilde{U}\cap \tilde{V}_t\cap E_{\delta}^{c}}\tilde{\chi}_{s_2}\wedge(\pi^*\chi)^{m-1},$$
                where $s_2$ is a fixed number.
                Hence we can choose $t_{\tilde{U}}$ sufficiently small such that
                $$ 
                \epsilon_5^{-1}\int_{\tilde{U}\cap \tilde{V}_t\cap E_{\delta}^{c}}\tilde{\chi}_{s_2}\wedge(\pi^*\chi)^{m-1}<\frac{\delta}{C_7}\frac{1}{10}\delta(\tilde{U}).$$
                Then we have $\int_{\tilde{U}\cap \tilde{V}_t} \tilde{\omega}_{s,t}\wedge(\pi^*\chi)^{m-1}>\delta_{\tilde{U}}$ for all $t<t_{\tilde{U}}$
                with $\delta_{\tilde{U}} = \frac{\delta}{C_7}\frac{4}{5}\delta(\tilde{U})$.

                By the Skoda-El Mir extension theorem and standard support theorem, 
                there exists $\beta_{\tilde{Y}}$ such that $\tilde{\Theta}_s\geq 2\beta_{\tilde{Y}}[\tilde{Y}]$, $\tilde{\Theta}_s \in [\tilde{\omega}_s]$.
                It is straightforward to verify that $\tilde{\Theta}_s$ satisfies the cone condition with respect to $\tilde{\chi}_s$ on $\tilde{Z}$. This proves Lemma~\ref{SecondConcentrate}.

                Let $\chi_Y$ be the K\"{a}hler metric in the cohomology class of $[Y]$ on $M$ and $\chi_{\tilde{Y}} = \pi^*\chi_Y |_{\tilde{Z}}$.
                Let $h_{\hat{E}}$ be the metric on $\hat{E}$ and $s_{\hat{E}}$ be the defining section.
                Take $\chi_{\hat{E}}$ be the form such that 
                $$
                \chi_{\hat{E}} = [ \hat{E} ] - \sqrt{-1} \partial \db \log | s_{\hat{E}} |_{h_{\hat{E}}}^{2}.$$
                Consider the current 
                $$
                \tilde{T}_s = \tilde{\Theta}_s - \beta_{\tilde{Y}}[\tilde{Y}] +(\beta_{\tilde{Y}}-s)\chi_{\tilde{Y}}  + (sC_4^{-1}+\beta_{\hat{E}})[\hat{E}] - \beta_{\hat{E}}\chi_{\hat{E}}\in [(1+s)\pi^*\omega].$$
                Decrease the upper bound $s_0$ for $s$, if necessary, so that $s_0<\beta_{\tilde{Y}}$. We can choose $\beta_{\hat{E}}$ sufficiently small so that there exists a sufficiently small constant $\epsilon_6$ satisfying
                $$(\beta_{\tilde{Y}}-s)\chi_{\tilde{Y}} - \beta_{\hat{E}} \chi_{\hat{E}} \geq \epsilon_6\tilde{\chi}_{s_0} \geq \epsilon_{6}\tilde{\chi}_{s}$$
                for any $0\leq s \leq s_0$.
                Here $\epsilon_6$ is independent of $s$.
                Then $\tilde{T}_s$ has positive mass and positive Lelong number on $\tilde{Y}\cup\hat{E}$
                and $(\beta_{\tilde{Y}}-s)\chi_{\tilde{Y}} - \beta_{\hat{E}} \chi_{\hat{E}} $ is K\"{a}hler.

                Let $x_1$ denote the largest real root of the first derivative of the degree-$m$ polynomial associated with the preceding equation on $\tilde Z$; it is distinct from the root of the ambient degree-$n$ polynomial introduced earlier. We can now choose an even smaller $s$ such that $(2s)^{-1} > \frac{x_1}{\epsilon_6}$. Then 
                \[
                (2s)^{-1}((\beta_{\tilde{Y}}-s)\chi_{\tilde{Y}} - \beta_{\hat{E}} \chi_{\hat{E}} ) \geq x_1\tilde{\chi}_s.
                \]
                By the definition of $x_1$,
                we know that   
                $(2s)^{-1}((\beta_{\tilde{Y}}-s)\chi_{\tilde{Y}} - \beta_{\hat{E}} \chi_{\hat{E}} )$
                satisfies the cone condition with respect to $\tilde{\chi}_s$.
                Since $\tilde{\chi}_s > \tilde{\chi}_0 = \pi^*\chi$, 
                $\tilde{\Theta}_s$ and $(2s)^{-1}((\beta_{\tilde{Y}}-s)\chi_{\tilde{Y}} - \beta_{\hat{E}} \chi_{\hat{E}} )$ 
                both satisfy the cone condition with respect to $\pi^*\chi$.
                
                Away from $\tilde{Y}\cup\hat{E}$, we have
                $$
                \frac{1}{1+s}\tilde{T}_s = \frac{1}{1+s}\tilde{\Theta}_s +\frac{s}{1+s} \Big( s^{-1}(\beta_{\tilde{Y}}-s)\chi_{\tilde{Y}} - s^{-1}\beta_{\hat{E}} \chi_{\hat{E}} ) \Big).$$
                By Proposition~\ref{SeveralResult}, the convexity of the cone implies that
                $$
                \frac{1}{1+s}\tilde{\Theta}_s +\frac{s}{1+s} \Big( (2s)^{-1}(\beta_{\tilde{Y}}-s)\chi_{\tilde{Y}} - (2s)^{-1}\beta_{\hat{E}} \chi_{\hat{E}} ) \Big)\in\bar\Upsilon^1(\pi^*\chi)
                $$
                away from $\tilde{Y}\cup\hat{E}$. Then Lemma~\ref{strictPRT} implies that
                $\frac{1}{1+s}\tilde{T}_s$ satisfies the strict cone condition  with respect to $\pi^*\chi$ away from $\tilde{Y}\cup\hat{E}$.
            
                Since $\pi(\tilde{Y})$ and $E:=\pi(\tilde{E})$ have strictly smaller dimension than $Z$, induction gives a neighborhood $\tilde{W}$ of $\pi(\tilde{Y})\cup E$ in $M$ and a smooth K\"ahler form $\tilde{\omega}_{\tilde{W}}$ satisfying the cone condition. As in the proof of Theorem~\ref{GluingThm}, since $\frac{1}{1+s}\tilde{T}_s$ has positive Lelong number on both $\tilde{Y}$ and $\tilde{E}$, we can glue $(\pi^*\tilde{\omega}_{\tilde{W}})|_{\tilde Z}$ with the local smoothing of $\frac{1}{1+s}\tilde{T}_s$ on subsets of $\tilde{Z}\backslash\tilde{E}$ to obtain
                \[
                \tilde{\Omega}=\pi^*\omega+\sqrt{-1}\partial\db\tilde{\psi}
                \]
                satisfying the cone condition on $\tilde{Z}\backslash\tilde{E}$. Since we can choose the local smoothing of $\frac{1}{1+s}\tilde{T}_s$ to avoid $\tilde{E}$, we can extend its pullback via $\pi^{-1}$ in the same manner as Lemma~\ref{SmoothExtention}, and also use $\tilde{\omega}_{\tilde{W}}$ as the natural extension of $(\pi^*\tilde{\omega}_{\tilde{W}})|_{\tilde Z}$. If the neighborhood is sufficiently small, we can still avoid the jump in the regularized maximum process near the boundary of the domain of the local defining functions and obtain a neighborhood $U$ of $Z$ and a K\"ahler metric $\omega_U=\omega_0+\sqrt{-1}\partial\db\psi_U$ satisfying the cone condition.
            \end{itemize}
        \end{proof}

\bibliographystyle{plain}
\bibliography{reference}

\bigskip
\begin{center}
\begin{minipage}{0.88\textwidth}
\small
\textbf{Gao Chen}\\[0.35em]
University of Science and Technology of China\\
No.\ 96 Jinzhai Road, Baohe District, Hefei, Anhui 230000, P.\,R.\ China\\[0.35em]
\href{mailto:chengao1@ustc.edu.cn}{\texttt{chengao1@ustc.edu.cn}}\\
\textbf{Sijie Nie}\\[0.35em]
University of Science and Technology of China\\
No.\ 96 Jinzhai Road, Baohe District, Hefei, Anhui 230000, P.\,R.\ China\\[0.35em]
\href{mailto:niesijie@mail.ustc.edu.cn}{\texttt{niesijie@mail.ustc.edu.cn}}\\
\textbf{Yulun Xu}\\[0.35em]
University of Toronto, 40 St. George Street, Toronto, ON, Canada\\[0.35em]
\href{mailto:yulun.1.xu@gmail.com}{\texttt{yulun.1.xu@gmail.com}}
\end{minipage}
\end{center}

\end{document}